\documentclass[a4paper,11pt]{amsart}
\usepackage{amsmath}
\usepackage{url,graphicx}
\usepackage{amssymb, color, pstricks-add, manfnt}
\usepackage{amsmath, amsthm,  amsfonts}
\usepackage{mathrsfs}
\usepackage{multirow}
\usepackage{cite}
\usepackage{enumerate}
\usepackage[none]{hyphenat}

\theoremstyle{plain}
\newtheorem{Theorem}{Theorem}[section]
\newtheorem{Lemma}{Lemma}[section]

\newtheorem{Corollary}{Corollary}[section]

\newtheorem{Example}{Example}[section]

\theoremstyle{remark}
\newtheorem{remark}{Remark}

\numberwithin{equation}{section}
\numberwithin{figure}{section}
\numberwithin{remark}{section}

\usepackage{autobreak}
\allowdisplaybreaks

\usepackage{hyperref}
\usepackage{cite}

\begin{document}
\begin{sloppypar}

	\title{On a class of Cauchy problems with applications in nonlinear partial differential equations}

	\author{Feida Jiang}
	\address{School of Mathematics and Shing-Tung Yau Center of Southeast University, Southeast University, Nanjing 211189, P.R. China; Shanghai Institute for Mathematics and Interdisciplinary Sciences, Shanghai, 200433, P.R.China}
	\email{jiangfeida@seu.edu.cn}	
	\author{Neil S. Trudinger$^*$}
	\address{Mathematical Sciences Institute, The Australian National University, Canberra ACT 0200, Australia}
	\email{neil.trudinger@anu.edu.au}
	\author{Qiao-Qiao Xu}
	\address{School of Mathematics and Shing-Tung Yau Center of Southeast University, Southeast University, Nanjing 211189, P.R. China}
	\email{qiaoqiaoxu@seu.edu.cn}
	\thanks{This work was supported by the National Natural Science Foundation of China (No. 12271093), Australian Research Council (No. DP230100499), the Jiangsu Provincial Scientific Research Center of Applied Mathematics (Grant No. BK20233002) and the Start-up Research Fund of Southeast University (No. 4007012503).}

	
	\date{\today}
	\thanks{*corresponding author}
	

	\begin{abstract}
	In this paper, we investigate the existence and nonexistence of entire solutions to a general class of Cauchy problems in the interval $[0, +\infty)$.
Our results provide a unified approach to proving Keller-Osserman type results for nonlinear partial differential equations in Euclidean space $\mathbb{R}^n$. As applications of the general framework, we present Keller-Osserman type results for two series of nonlinear equations: a series of $k$-Hessian type equations and a new series of $\Pi_k$-Hessian type equations. These results are also proved for the more general $p$-Hessian matrices, with $p>1$, and degenerate inhomogeneous terms. 
         \end{abstract}


\maketitle

\section{Introduction}\label{Section 1}

\sloppy{}

In this paper, we study a class of Cauchy problems of the following form:
	\begin{equation}\label{1.1}
	\left\{
	\begin{aligned}
	&v'(r)= C  r^{-q}  \left(  \int_{0}^{r} s^{\tau-1} g (v(s))ds \right )^{\frac{1}{\theta}}, \quad r\in (0, +\infty), \\
	&v(0)=a,
	\end{aligned}
	\right.
	\end{equation}
for any given initial constant $a$, and constants $q\ge 0$, $\tau >0$, $\theta> 0$, $C>0$, where $g: \mathbb{R} \to (0,+\infty)$ is a non-decreasing continuous function, and $v: [0, +\infty) \to [a, +\infty)$ is an unknown function.

We are interested in proving both local existence and entire existence of solutions of the Cauchy problem \eqref{1.1}, which will be applied to various classes of nonlinear partial differential equations, extending the basic work of  Bao and Ji  \cite{JB2010} on $k$-Hessian equations. We now formulate our first main result, the local existence and regularity for \eqref{1.1}, under the condition $\tau > q\theta$.
\begin{Theorem}[Local existence and regularity]\label{Th 1.1}
Assume that $g: \mathbb{R} \to (0,+\infty)$ is a non-decreasing continuous function, the constants $q\ge 0$, $\theta> 0$, $C>0$ and $\tau > q\theta$. Then for any constant $a$, there exists a positive number $R$ such that the Cauchy problem \eqref{1.1} admits a solution $v \in C^1[0, R)\cap C^2(0, R)$ with $v'(0)=0$, $v'(r)>0$ and $v''(r)>0$ for $r\in (0, R)$.
Moreover, the following properties hold:
\begin{itemize}
\item[(i)] If $\tau> \theta (q+1)$, then $v\in C^2[0, R)$ with $v''(0)=0$; 

\item[(ii)] If $\tau= \theta (q+1)$, then $v\in C^2[0, R)$ with $v''(0)=C(\frac{g(a)}{\tau})^{\frac{1}{\theta}}$; 

\item[(iii)] If $\tau < \theta (q+1)$, then $v \in C^2(0, R)\cap W^{2, \delta}(0, R)$, where $\delta \in (1, \frac{\theta}{\theta (q+1)-\tau})$.
\end{itemize}
\end{Theorem}

For the Cauchy problem \eqref{1.1}, we introduce the following integral condition to guarantee the  existence of entire solutions:
	\begin{equation}\label{KO}
	\int^{+\infty} \left ( \int_0^t g(s)ds \right )^{-\frac{1}{\theta+1}} dt = + \infty,
	\end{equation}
where we omit the lower limit of integral to allow any positive number.

We then formulate our second main result, which shows that condition \eqref{KO} is both necessary and sufficient for the existence of entire solutions of the Cauchy problem \eqref{1.1}, under the condition $\tau \ge q\theta+1$.	
\begin{Theorem}\label{Th 1.2}
Assume that $g: \mathbb{R} \to (0,+\infty)$ is a non-decreasing continuous function, the constants $q\ge 0$, $\theta> 0$, $C>0$ and $\tau \ge \theta q+1$. Then the Cauchy problem \eqref{1.1} has a solution $v \in C^1[0, +\infty)\cap C^2(0, +\infty)$ with $v'(0)=0$, $v'(r)>0$ and $v''(r)>0$ for $r\in (0, +\infty)$ and for any initial value $a$, if and only if the condition \eqref{KO} holds. Moreover, the following properties hold:
\begin{itemize}
\item[(i).] If $\tau> \theta (q+1)$, then $v\in C^2[0, +\infty)$ with $v''(0)=0$; 

\item[(ii).] If $\tau= \theta (q+1)$, then $v\in C^2[0, +\infty)$ with $v''(0)=C(\frac{g(a)}{\tau})^{\frac{1}{\theta}}$; 

\item[(iii).] If $\tau < \theta (q+1)$, then $v \in C^2(0, +\infty) \cap W_{loc}^{2, \delta}(0, +\infty)$, where $\delta \in (1, \frac{\theta}{\theta (q+1)-\tau})$.
\end{itemize}
\end{Theorem}

In Theorem \ref{Th 1.2}, if the entire solution $v(r)$ exists, then it is unbounded, namely $\lim\limits_{r\to +\infty} v(r)= +\infty$.
This is because $v'(r)>0$ and $v''(r)>0$ for $r\in (0, +\infty)$.

Note that an alternative assumption for $g$ in Theorem \ref{Th 1.2} is that $g: \mathbb{R} \to [0, \infty)$ is nonnegative and non-decreasing satisfying $g(0)=0$ and $g(t)>0$ for $t>0$. In this case, we require the initial value $a\ge 0$, and only obtain a nonnegative entire solution $v$, see Remark \ref{Re 6.6}. 

Since condition \eqref{KO} is both necessary and sufficient for the existence of entire solutions of the Cauchy problem \eqref{1.1}, we may instead consider the boundary blow-up problem in bounded domains if \eqref{KO} is not satisfied; see \cite{JDJ2025} for an example.

\begin{remark}\label{Remark 1.1}
Note that the assumption on $\tau$ for the local existence in Theorem \ref{Th 1.1} is $\tau > \theta q$, while the assumption on $\tau$ for the global existence in Theorem \ref{Th 1.2} is $\tau \ge \theta q+1$. When $\theta \ge 1$ in Theorem \ref{Th 1.2}, then $\tau$ in (i) or (ii) automatically satisfies $\tau \ge \theta q+1$. When $\theta > 1$ in Theorem \ref{Th 1.2}, $[\theta q+1, \theta (q+1))$ is the proper interval for $\tau$ in (iii). When $0<\theta<1$, under the assumption $\tau \ge \theta q+1$, only case (i) can occur. Therefore, Theorem \ref{Th 1.2}  is different in applications to the $\theta \ge 1$ and $0<\theta <1$ cases.
\end{remark}

Next, we formulate our third main result, which is the entire existence and nonexistence for the Cauchy problem \eqref{1.1} under the condition $\theta q< \tau < q\theta+1$. In this part, the necessary and  sufficient conditions for the entire existence of \eqref{1.1} are different.
\begin{Theorem}\label{Th 1.3}
Assume that $g: \mathbb{R} \to (0,+\infty)$ is  non-decreasing and continuous, the constants $q\ge 0$, $\theta> 0$, $C>0$ and $\theta q  < \tau <  q\theta+1$. 

(a). If $g$ satisfies
	\begin{equation}\label{1.3}
	\int^{+\infty} \left ( \int_0^t g^{\kappa}(s)ds \right )^{-\frac{1}{\kappa \theta+1}} dt = + \infty,
	\end{equation}
where $\kappa =\frac{1-\epsilon}{\tau -\theta q - \epsilon}$ for some constant $\epsilon\in (0, \tau-\theta q)$,
then the Cauchy problem \eqref{1.1} has a solution $v \in C^1[0, +\infty)\cap C^2(0, +\infty)$ with $v'(0)=0$, $v'(r)>0$ and $v''(r)>0$ for $r\in (0, +\infty)$ and for some initial value $a$. Moreover, the following properties hold:
\begin{itemize}
\item[(i).] If $\tau> \theta (q+1)$, then $v\in C^2[0, +\infty)$ with $v''(0)=0$; 

\item[(ii).] If $\tau= \theta (q+1)$, then $v\in C^2[0, +\infty)$ with $v''(0)=C(\frac{g(a)}{\tau})^{\frac{1}{\theta}}$; 

\item[(iii).] If $\tau < \theta (q+1)$, then $v \in C^2(0, +\infty) \cap W_{loc}^{2, \delta}(0, +\infty)$, where $\delta \in (1, \frac{\theta}{\theta (q+1)-\tau})$.
\end{itemize}

(b). If the Cauchy problem \eqref{1.1} has a solution $v \in C^1[0, +\infty)\cap C^2(0, +\infty)$ satisfying the properties in cases (i), (ii) or (iii) in (a), with $v'(0)=0$, $v'(r)>0$ and $v''(r)>0$ for $r\in (0, +\infty)$ and for some initial value $a$, then the condition \eqref{KO} holds.
\end{Theorem}

Note that the conditions \eqref{KO} and \eqref{1.3} are equivalent in several important applications of Theorem \ref{Th 1.3}, see Corollaries \ref{Co 1.1} and \eqref{Co 1.2}. In such particular cases, Theorem \ref{Th 1.3} still provides necessary and sufficient conditions for the entire existence.

\begin{remark}\label{Remark 1.2}
The assumption of $\tau$ in Theorem \ref{Th 1.3} is $\theta q< \tau < q\theta+1$. 
When $\theta \ge 1$ in Theorem \ref{Th 1.3}, the $\tau$ in (ii), and the $\tau$ in (iii) satisfying $\tau>\theta q$, satisfy $\theta q<\tau < \theta q+1$ automatically. When $0<\theta <1$ in Theorem \ref{Th 1.3}, $(\theta (q+1), \theta q+1)$ is the proper interval for $\tau$ in (i). When $\theta>1$, under the assumption $\theta q<\tau < \theta q+1$, only case (iii) can occur. Therefore, Theorem \ref{Th 1.3} behaves differently in applications to the $0<\theta \le 1$ and $\theta > 1$ cases.
\end{remark}

The Cauchy problem, a cornerstone in the field of differential equations, has been extensively studied due to its wide applicability in various scientific disciplines, including physics, engineering, and economics. We recall some historical developments of the Cauchy problems with applications to the entire existence and nonexistence of partial differential equations. The problem \eqref{1.1} with $C=1$, $q=n-1$, $\tau=n$, $\theta=1$ and $g=f$ was studied in \cite{K1957} with applications in the semilinear elliptic equations with Laplacian operator
\begin{equation}\label{1.4}
\Delta u = f(u), \quad {\rm in} \ \mathbb{R}^n. 
\end{equation}
The condition \eqref{KO} with $g=f$ and $\theta =1$ is the well-known Keller-Osserman condition of \eqref{1.4}, which is a necessary and sufficient condition for the existence of entire solutions, see \cite{K1957, O1957}. For the Laplacian equation on the complete Riemannian manifold with Ricci curvature bounded from below by a constant, the convergence of the integral in \eqref{KO} with $g=f$ and $\theta =1$ implies the nonexistence of the $C^2$ solution in \cite{CY1975}. The Keller-Osserman conditions for \eqref{1.4} with the Laplacian operator replaced by $p$-Laplacian operator and $k$-Hessian operator were treated in \cite{NU1997} and \cite{JB2010}, respectively. In conjunction with the theory of  $k$-Hessian measures, the second author and X.-J. Wang introduced the $p$-$k$-Hessian operator $S_k[D(|Du|^{p-2}Du)]$ in \cite{TW1999}. The Keller-Osserman condition for \eqref{1.4} with the Laplacian operator was extended to $p$-$k$-Hessian operators in \cite{BF2022}.

The significance of our work lies in the unification of the general framework of the Cauchy problem \eqref{1.1}, which implies two more general classes of nonlinear partial differential equations in the entire Euclidean space $\mathbb{R}^n$. 
The results in \cite{K1957, O1957, NU1997, JB2010, BF2022} are special cases of Theorems \ref{Th 1.1} and \ref{Th 1.2}.
Apart from the applications to equations with the basic right hand side term $f(u)$, we allow the right hand side term also depend on the gradient $Du$ and the variable $x$. Our results are then applicable to degenerate or singular equations and those with more complex structures, thus broadening the scope of existing literature. Our main applications of Theorems \ref{Th 1.1} and \ref{Th 1.2} are the following degenerate or singular $k$-Hessian type equations
	\begin{equation}\label{1.5}
	S_k^{\frac{1}{k}} [D(|Du|^{p-2}Du)] = |x|^\alpha |Du|^\beta f(u), \quad {\rm in} \ \mathbb{R}^n,
	\end{equation}
and their counterparts, degenerate or singular $\Pi_k$-Hessian type equations
	\begin{equation}\label{1.6}
	\Pi_k^{\frac{1}{C_n^k}} [D(|Du|^{p-2}Du)] = |x|^\alpha |Du|^\beta f(u), \quad {\rm in} \ \mathbb{R}^n,
	\end{equation}
where $p>1$, $k=1, \cdots, n$, $\alpha$ and $\beta$ are given constants. The definitions of the operators $S_k$ and $\Pi_k$ together with the Keller-Osserman theorems for equations \eqref{1.5} and \eqref{1.6} are presented in Sections \ref{Section 4} and \ref{Section 5}, respectively.

\vspace{2mm}

The main features of this paper can be described as follows.

\vspace{1mm}

(i). By investigating the general Cauchy problem \eqref{1.1} in a uniform package, we are able to extend the theory for $k$-Hessian equations in \cite{K1957, O1957, NU1997, JB2010, BF2022} to the degenerate or singular $k$-Hessian type equations \eqref{1.5} with the right hand side term $|x|^\alpha |Du|^\beta f(u)$. The theory also applies to the completely new series of $\Pi_k$-Hessian type equations \eqref{1.6}, which includes the $(n-1)$-Monge-Amp\`ere case in \cite{JJL2024} as a special case. As far as we know, the only previously known result with such right hand side term is in \cite{FPR2010}, where the equation
\begin{equation}\label{1.7}
{\rm div} (|Du|^{p-2}Du)= |x|^\alpha |Du|^\beta f(u)
\end{equation}
is studied. Note that equation \eqref{1.7} is the particular case of equation \eqref{1.5} when $k=1$, and also the particular case of equation \eqref{1.6} when $k=n$. 

\vspace{1mm}

(ii). In \cite{JB2010} and \cite{BF2022}, equation \eqref{1.5} in the cases $p=2$, $\alpha=\beta=0$ and $p>2$, $\alpha=\beta=0$ was studied, respectively. As far as we know, there were no such results for the $1<p <2$ case of equation \eqref{1.5}. The results of this paper when $1<p<2$ are new.

\vspace{1mm}

(iii). The interval for the parameter $\theta$ in the existing literature was $[0, +\infty)$. For our general equations \eqref{1.5} and \eqref{1.6}, the restriction $\theta \in (0,1)$ occurs in many situations; see Remark \ref{Re 6.1}. For this reason, we prove the case when $\theta \in (0,1)$ in Step 2 of the proof of Theorem \ref{Th 1.1} via a different method from the case when $\theta \in [1, +\infty)$. Accordingly, the results of this paper in the case when $0<\theta <1$ are also new.

\vspace{1mm}

(iv). We use the assumption $\tau > \theta q$ for the local existence in Theorem \ref{Th 1.1} and the assumption $\tau \ge \theta q +1$ for global existence in Theorem \ref{Th 1.2}.  Then we also investigate the global existence of problem \eqref{1.1} when $\theta q<\tau < \theta q +1$ in Theorem \ref{Th 1.3}, which supplements the case when $-1<\alpha < \frac{1}{k} -1$ in Theorem \ref{Th 4.1} and the case when $-1< \alpha < \frac{k}{n}-1$ in Theorem \ref{Th 5.1}. It should be noted that these supplementary cases have not been considered before.

\vspace{1mm}

(v). When applying Theorems \ref{Th 1.2} and \ref{Th 1.3} to the $k$-Hessian type equations \eqref{1.5} and the $\Pi_k$-Hessian equations \eqref{1.6} with exponential nonlinearities and power type nonlinearities, condition \eqref{1.3} is equivalent to condition \eqref{KO}. Therefore, both Theorem  \ref{Th 1.2} and Theorem \ref{Th 1.3} provide necessary and sufficient conditions towards the entire existence of these equations with such nonlinearities, see Corollaries \ref{Co 1.1} and \ref{Co 1.2}.

\vspace{1mm}

This paper is organized as follows. In Section \ref{Section 2}, we prove the local existence and regularity of the Cauchy problem \eqref{1.1}. In Section \ref{Section 3},  we study the entire existence and nonexistence of the Cauchy problem \eqref{1.1}. We show that condition \eqref{KO} is a necessary and sufficient condition and prove Theorem \ref{Th 1.2}. In Sections \ref{Section 4} and \ref{Section 5}, we apply Theorem \ref{Th 1.2} to two types of nonlinear partial differential equations, $k$-Hessian type equations and $\Pi_k$-Hessian type equations, respectively.  In Section \ref{Section 5}, we make several final remarks. Some explicit examples of the functions $f$ for \eqref{1.5} and \eqref{1.6}, and the entire subsolutions of $\Pi_k$-Hessian type equations \eqref{1.6} are presented at the end.

\vspace{3mm}



\section{Local existence and regularity of the Cauchy problem}\label{Section 2}

\sloppy{}

In this section, we treat the Cauchy problem \eqref{1.1} and present the proof of Theorem \ref{Th 1.1}. The local existence of \eqref{1.1} is proved by Euler's method, and the regularity of the local solution of \eqref{1.1} is discussed for different ranges of $\tau$. At the end of this section, a further property of the local solution of \eqref{1.1} is also discussed when $\tau = \theta (q+\frac{1}{p-1})$ for some $p>1$.

\begin{proof}[Proof of Theorem \ref{Th 1.1}] We first construct the local solution $v$ by Euler's method, which is divided into three steps.

	\textit{Step 1: We define the piecewise linear function $\psi$}.
	We define a functional $F(\cdot, \cdot)$ on the rectangle
	\begin{equation*}
		\mathcal{R}:=[0,l]\times\{v \in C^0[0,l]: a\leq v <a+h\}
	\end{equation*}
	by
	\begin{equation}\label{F def}
		G(r,v):=C  r^{-q}  \left(  \int_{0}^{r} s^{\tau-1} g (v(s))ds \right )^{\frac{1}{\theta}},
	\end{equation}	
	where $l$ and $h$ are small positive constants.
	Then $v^{\prime}(r)=G(r,v(r))$
	and $G(r,v)>0$ for $r>0$.
	We define a piecewise linear function $\psi(\cdot)$ on $[0,l]$, by
	\begin{equation}\label{new aaa}
		\left\{
		\begin{aligned}
			&\psi(r)=a, \quad r_0\leq r\leq r_1,\\
			&\psi(r)=\psi(r_{i-1})+G\left( r_{i-1},\psi(r_{i-1})\right) (r-r_{i-1}), \quad r_{i-1} < r \leq r_i,\quad i=2,\cdots,m,
		\end{aligned}
		\right.
	\end{equation}
	where $0=r_0<r_1<\cdots<r_m=l$ and $m\in \mathbb{N}$.
		
	Then for sufficiently small $l$, we have 
	\begin{equation}\label{2.10}
	a \le \psi(r) < a+h, \quad r\in [0,l].
	\end{equation}
	To prove \eqref{2.10}, since $g$ is non-decreasing and $\tau > q\theta$, we have
	\begin{align}
	F(r, \psi) \le & C  r^{-q}  \left(  \int_{0}^{r} s^{\tau-1} ds \right )^{\frac{1}{\theta}}g^{\frac{1}{\theta}} (\psi(l)) \notag\\
		        = & C \left( \frac{1}{\tau} \right)^{\frac{1}{\theta}} r^{\frac{\tau}{\theta}-q} g^{\frac{1}{\theta}}(\psi(l))  \label{2.11} \\
	   	      \le & C \left( \frac{1}{\tau} \right)^{\frac{1}{\theta}} l^{\frac{\tau}{\theta}-q} g^{\frac{1}{\theta}}(\psi(l)) \notag
	\end{align}
	for $r\in [0,l]$.
	Hence from \eqref{new aaa} and \eqref{2.11}, we obtain
	\begin{equation*}
	a \le \psi(r) \le a + \max\limits_{0\le r\le l} G(r, \psi(r)) l \le a + C \left( \frac{1}{\tau} \right)^{\frac{1}{\theta}} l^{\frac{\tau}{\theta}+1-q} g^{\frac{1}{\theta}}(\psi(l)), \quad r\in [0,l],
	\end{equation*}
	so that if $h$ is fixed, we can choose $l$ sufficiently small such that \eqref{2.10} holds. Namely, $(r, \psi)\in \mathcal{R}$.

	\vspace{2mm}
	
	\textit{Step 2: We prove that $\psi$ is an $\varepsilon$-approximation solution of \eqref{1.1}}. To do this, we need to prove that for any small $\varepsilon > 0$, we can choose appropriate points $\{r_i\}_{i=1,\cdots,m}$, so that
	\begin{equation}\label{2.13}
		\left| \psi^{\prime}(r)-G(r,\psi(r))\right| <\varepsilon, \quad r\in [0,l].
	\end{equation}
	It follows from \eqref{F def} and $\tau>\theta q$ that
	$$\lim_{r\to 0}G(r,\psi(r))=0.$$
	Then for each $\varepsilon>0$, there exists $\bar r \in (0,l)$ such that for $0\leq r< \bar r$, we have
	$$G(r,\psi(r))<\varepsilon.$$
	We can choose $r_1=\bar r$, which gives
	\begin{equation}\label{2.14}
		\left| \frac{d\psi(r)}{dr}-G(r,\psi(r))\right| =\left| G(0, \psi(0))-G(r,\psi(r))\right| <\varepsilon,	\quad 0\le r< \bar r.
	\end{equation}
	
For $\bar r\leq r \le l$, without loss of generality, we assume $r_{i-1} \le r \le r_i$ for some $i=2,\cdots,m$.
We divide the discussion into two cases $\theta \ge 1$ and $0< \theta < 1$. 

	{\bf Case 1.} $\theta \ge 1$. In this case, we have
		\begin{align}
			&\left| \frac{d\psi(r)}{dr}-G(r,\psi(r))\right| \notag\\
			=&\left| G(r_{i-1},\psi (r_{i-1}))-G(r,\psi(r))\right| \notag\\
			=&C \left| 	r_{i-1}^{-q}\left( \int_{0}^{r_{i-1}}s^{\tau-1}g(v(s))ds\right)^\frac{1}{\theta}\!\!-\!	r^{-q}\left( \int_{0}^{r}s^{\tau-1}g(v(s))ds\right)^\frac{1}{\theta}\right| \label{2.15}\\
			\leq& C \left|  r_{i-1}^{-q\theta}\int_{0}^{r_{i-1}}s^{\tau-1}g(v(s))ds-r^{-q\theta}\int_{0}^{r}s^{\tau-1}g(v(s))ds\right|^\frac{1}{\theta}.\notag 
		\end{align}
Since $\int_{0}^{r}s^{\tau-1}g(v(s))ds = (\int_{0}^{r_{i-1}} + \int_{i-1}^{r})s^{\tau-1}g(v(s))ds$, we have from \eqref{2.15} that
		\begin{align}
			&\left| \frac{d\psi(r)}{dr}-G(r,\psi(r))\right| \notag\\
			\leq& C \left (  \left |r_{i-1}^{-q\theta}-r^{-q\theta}\right | \int_{0}^{r_{i-1}}s^{\tau-1}g(v(s))ds + r^{-q\theta}\int_{r_{i-1}}^{r}s^{\tau-1}g(v(s))ds\right )^\frac{1}{\theta} \notag \\
			 \leq& C g^{\frac{1}{\theta}}(a+h)\left (  \left |r_{i-1}^{-q\theta}-r^{-q\theta}\right | \int_{0}^{r_{i-1}}s^{\tau-1}ds + r^{-q\theta}\int_{r_{i-1}}^{r}s^{\tau-1}ds\right )^\frac{1}{\theta} \label{2.16} \\
			 \leq& C \left(\frac{1}{\tau}\right)^{\frac{1}{\theta}} g^{\frac{1}{\theta}}(a+h)\left ( l^\tau \left |r_{i-1}^{-q\theta}-r^{-q\theta}\right |  + {\bar r}^{-q\theta} |r^\tau - r_{i-1}^\tau | \right )^\frac{1}{\theta} \notag \\
			 \leq& C \left(\frac{1}{\tau}\right)^{\frac{1}{\theta}} g^{\frac{1}{\theta}}(a+h)\left ( l^\tau {\bar r}^{-q\theta -1}q\theta + \tau l^{\tau-1}{\bar r}^{-q\theta } \right )^{\frac{1}{\theta}} |r-r_{i-1}|^{\frac{1}{\theta}}, \notag 
		\end{align}
		where the mean value theorem is used in the last inequality.
	For the above $\varepsilon>0$, we take
	 \begin{equation}\label{2.17}
	 0<\delta_1 < \frac{\tau{\bar r}^{q\theta +1}}{l^{\tau-1}(q\theta l + \tau \bar r)g(a+h)}\left( \frac{\varepsilon}{C}\right)^\theta.
	 \end{equation}
	Then for $\max\limits_{2\leq i \leq m}\left|r-r_{i-1} \right| <\delta_1$, we have from \eqref{2.16} and \eqref{2.17} that
	 \begin{equation}\label{2.14'}
	 \left| \frac{d\psi(r)}{dr}-G(r,\psi(r))\right|<\varepsilon, \quad r\in [\bar r,l].
	 \end{equation}		
	 
	{\bf Case 2.} $0<\theta <1$. In this case, we have
	 \begin{align}
	  &\left| \frac{d\psi(r)}{dr}-G(r,\psi(r))\right| \notag\\
	=&\left| G(r_{i-1},\psi (r_{i-1}))-G(r,\psi(r))\right| \notag\\
	=&C \left| 	r_{i-1}^{-q}\left( \int_{0}^{r_{i-1}}s^{\tau-1}g(v(s))ds\right)^\frac{1}{\theta}\!\!-\!	r_{i-1}^{-q}\left( \int_{0}^{r}s^{\tau-1}g(v(s))ds\right)^\frac{1}{\theta}\right. \notag\\
	  & \left. + 	r_{i-1}^{-q}\left( \int_{0}^{r}s^{\tau-1}g(v(s))ds\right)^\frac{1}{\theta} - \!	r^{-q}\left( \int_{0}^{r}s^{\tau-1}g(v(s))ds\right)^\frac{1}{\theta} \right | \notag \\
	=&Cr_{i-1}^{-q}  \left| \left( \int_{0}^{r_{i-1}}s^{\tau-1}g(v(s))ds\right)^\frac{1}{\theta}\!\!-\!	\left( \int_{0}^{r}s^{\tau-1}g(v(s))ds\right)^\frac{1}{\theta}\right| \notag \\
	  & + C \left( \int_{0}^{r}s^{\tau-1}g(v(s))ds\right)^\frac{1}{\theta} \left | r_{i-1}^{-q} - r^{-q} \right |. \notag
	 \end{align}	
	 Using the mean value theorem, we have
	 \begin{align}
	     &\left| \frac{d\psi(r)}{dr}-G(r,\psi(r))\right| \notag\\
	\le & C {\bar r}^{-q} \frac{1}{\theta} \left( \int_{0}^{r}s^{\tau-1}g(v(s))ds\right)^{\frac{1}{\theta}-1} \int_{r_{i-1}}^{r}s^{\tau-1}g(v(s))ds \notag \\
	     & + Cq {\bar r}^{-q-1} \left( \int_{0}^{r}s^{\tau-1}g(v(s))ds\right)^\frac{1}{\theta} |r-r_{i-1}| \notag \\
	\le & C {\bar r}^{-q} \frac{1}{\theta} g^{\frac{1}{\theta}}(a+h) \left ( \frac{1}{\tau} \right )^{\frac{1}{\theta}-1} l^{\tau(\frac{1}{\theta}-1)}l^{\tau-1} |r-r_{i-1}| \label{2.18} \\
	     & + Cq {\bar r}^{-q-1} g^{\frac{1}{\theta}}(a+h)\left ( \frac{1}{\tau} \right )^{\frac{1}{\theta}} l^{\frac{\tau}{\theta}} |r-r_{i-1}| \notag \\
	   =&  C {\bar r}^{-q-1} \frac{1}{\theta} \left ( \frac{1}{\tau} \right )^{\frac{1}{\theta}} l^{\frac{\tau}{\theta}-1}  g^{\frac{1}{\theta}}(a+h) ( \tau \bar r + \theta q l ) |r-r_{i-1}|. \notag
	 \end{align}
	 For the above $\varepsilon>0$, we take
	 \begin{equation}\label{2.19}
	 0<\delta_2 < \frac{\theta {\bar r}^{q+1} \tau^{\frac{1}{\theta}} \varepsilon}{C l^{\frac{\tau}{\theta}-1}  g^{\frac{1}{\theta}}(a+h) ( \tau \bar r + \theta q l )}.
	 \end{equation}
	 Then for $\max\limits_{2\leq i \leq m}\left|r-r_{i-1} \right| <\delta_2$, we have from \eqref{2.18} and \eqref{2.19} that
	 \begin{equation}\label{2.14''}
	 \left| \frac{d\psi(r)}{dr}-G(r,\psi(r))\right|<\varepsilon, \quad r\in [\bar r,l].
	 \end{equation}	
	 
\vspace{2mm}
	
	Combining \eqref{2.14}, \eqref{2.14'} and \eqref{2.14''}, we obtain \eqref{2.13}, which implies that $\psi$ is an $\varepsilon$-approximation solution of \eqref{1.1}.
		
	\vspace{2mm}	
	\textit{Step 3: We deduce the solvability of \eqref{1.1} from our construction of $\psi$}.
	Let $\{\varepsilon_j\}_{j=1}^{\infty}$ be a sequence of positive constants converging to $0$. Assume that sequence $\{\psi_j\}$ is the $\varepsilon_j$-approximation solution of \eqref{1.1} on $[0,l]$, that is,
	\begin{equation}\label{2.20}
		\left| \frac{d\psi_j(r)}{dr}-G(r,\psi_j(r))\right| <\varepsilon_j, \quad r\in [0,l].
	\end{equation} 	
	It is easy to see that
	\begin{equation*}
		\left| \psi_j(r^{\prime \prime })-\psi_j(r^{\prime})\right| <\max_{1\leq i \leq m}G(r_i,\psi(r_i))|r^{\prime\prime}-r^{\prime}|\leq M|r^{\prime\prime}-r^{\prime}|,
	\end{equation*}
	for any $r^{\prime}, r^{\prime \prime } \in [0,l]$, whence
	 the sequence $\{\psi_j\}$ is equicontinuous and uniformly bounded.
	By the Arzel\` a-Ascoli theorem, there exists a uniformly convergent subsequence, which we still  denoted as $\{\psi_j\}$.
	Assume $\lim\limits_{j\to +\infty}\psi_j=v$. 
	Since $\psi_j$ is continuous in $[0, l]$, we know that $v$ is continuous in $[0, l]$.
	Since $\psi_j(0)=0$, we have $v(0)=0$. Setting
	\begin{equation}\label{new 1 in Sec 2}
		\triangle_j(r):=\frac{d\psi_j(r)}{dr}-G(r,\psi_j(r)),
	\end{equation}
	from \eqref{2.20}, we have
	\begin{equation}\label{new 2 in Sec 2}
	|\triangle_j(r)|<\varepsilon_j, \quad   r\in[0,l].
	\end{equation}
	Integrating \eqref{new 1 in Sec 2} and using \eqref{new 2 in Sec 2}, we let $j\to +\infty$ and obtain
	\begin{equation}\label{2.21}
		\begin{split}
			v(r)=\lim_{j\to +\infty}\psi_j(r)
			&=a+\lim_{j\to +\infty}\int_0^r\left[  G(s,\psi_j(s))+\triangle_j(s)\right] ds\\
			&=a+\int_0^r \lim_{j\to +\infty}\left[  G(s,\psi_j(s))+\triangle_j(s)\right] ds\\
			&=a+\int_0^r G(s,v(s))ds.
		\end{split}
	\end{equation}
	Since $v$ is continuous in $[0, l]$, by \eqref{2.21}, $v$ is continuously differentiable in $(0,l]$.
	Differentiating \eqref{2.21}, we can see that $v$ satisfies \eqref{1.1} on $[0,l]$.
	
	\vspace{2mm}
	
	From the above three steps, we get a solution $v$ of \eqref{1.1} on $[0,l]$.
	In fact, we can extend $[0,l]$ to a larger interval. Without loss of generality, we assume that there exists a maximum existence interval $[0,R)$ such that $v$ satisfies \eqref{1.1} on $[0,R)$ with $v(r)\to \infty$ as $r\to R$. Then $v \in C[0, R)\cap C^1(0, R)$.
	
	\vspace{2mm}
	Next, we will check the differentiability of $v$ at $0$ and the value of $v'(0)$. By l'H\^opital's rule,  we have from \eqref{1.1} that
	\begin{equation}\label{lim v(r)}
			\lim_{r\to 0}v^{\prime}(r)
			=C \left( \lim_{r\to 0} \frac{\int_{0}^{r}s^{\tau-1}g(v(s)) ds}{r^{\theta q}}  \right) ^\frac{1}{\theta}
			=C \left( \frac{g(a)}{\theta q} \lim_{r\to 0} r^{\tau-\theta q} \right)^{\frac{1}{\theta}}=0,
	\end{equation}
	where $v(0)=a$ and $\tau>\theta q$ are used. Using \eqref{2.21}, $v(0)=a$ and l'H\^opital's rule again, we have
	\begin{equation}\label{v'(0)}
	\lim_{r\to 0} \frac{v(r) - v(0)}{r-0} = \lim_{r\to 0} \frac{\int_{0}^r G(s, v(s))ds}{r} = \lim_{r\to 0} G(r, v(r))= \lim_{r\to 0}v^{\prime}(r)=0,
	\end{equation}
	where \eqref{lim v(r)} is used in the last equality. From \eqref{lim v(r)} and \eqref{v'(0)}, we see that $v$ is continuously differentiable at $0$ and $v'(0)=0$. The monotonicity of $v$ is obvious from \eqref{1.1} and the positivity of $g$. Namely that $v'(r)>0$ for $r\in (0, R)$.
	
	Actually, $v\in C^2(0, R)$ is obvious from the form of $v'(r)$ in \eqref{1.1}.
	Differentiating the equation in \eqref{1.1}, we get
	\begin{equation}\label{2diff}
	\begin{aligned}
			v^{\prime\prime}(r)
			 &  =  \frac{C}{\theta} g(v(r)) r^{-(q+1-\tau)} \left ( \int_{0}^{r} s^{\tau-1} g (v(s))ds\right)^{\frac{1}{\theta}-1} 
			          -Cq r^{-(q+1)} \left ( \int_{0}^{r} s^{\tau-1} g (v(s))ds\right)^{\frac{1}{\theta}} \\
			 & = C \left ( \int_{0}^{r} s^{\tau-1} g (v(s))ds\right)^{\frac{1}{\theta}-1} \left ( \frac{g(v(r))}{\theta r^{q+1-\tau}} - \frac{q \int_{0}^{r} s^{\tau-1} g (v(s))ds}{r^{q+1}} \right ),
	\end{aligned}
	\end{equation}	
	for $r\in (0, R)$. From the monotonicity of $v$ and $g$, we deduce from \eqref{2diff} that
	\begin{equation}\label{convexity}
	v''(r) \ge C \left(\frac{1}{\theta}-\frac{q}{\tau}\right) r^{\tau - (q+1)} g(v(r)) \left ( \int_{0}^{r} s^{\tau-1} g (v(s))ds\right)^{\frac{1}{\theta}-1} >0,
	\end{equation}
	where the positivity of $g$ and $\tau > \theta q$ is used in the last inequality.	
	
	Up to now, we have proved that $v \in C^1[0, R)\cap C^2(0, R)$ with $v'(0)=0$, $v'(r)>0$ and $v''(r)>0$ for $r\in (0, R)$.
	
	\vspace{2mm}

	Finally, we show the assertions (i), (ii) and (iii). By the definition of the second order derivative of $v$ at $0$, from \eqref{1.1} we have
	\begin{equation}\label{v''(0)}
	\begin{aligned}
	v''(0) & = \lim_{r\to 0} \frac{v'(r)-v'(0)}{r-0} = C \lim_{r\to 0} \frac{ \left(  \int_{0}^{r} s^{\tau-1} g (v(s))ds \right )^{\frac{1}{\theta}}}{r^{q+1}} \\
		& = C \left( \frac{g(a)}{\theta(q+1)} \right)^{\frac{1}{\theta}} \lim_{r\to 0} r^{\tau - \theta (q+1)}.
	\end{aligned}
	\end{equation}
	Using l'H\^opital's rule, we can calculate from \eqref{2diff} that 
	\begin{equation}\label{limv''}
	\begin{aligned}
	   & \lim_{r \to 0} \frac{v''(r)}{r^{\frac{1}{\theta}[\tau - \theta (q+1)]}} \\ 
	= & C \lim_{r \to 0} \left( \frac{\int_{0}^{r} s^{\tau-1} g (v(s))ds}{r^\tau}\right)^{\frac{1}{\theta} - 1} \left ( \frac{g(v(r))}{\theta} - \frac{q \int_{0}^{r} s^{\tau-1} g (v(s))ds}{r^{\tau}} \right ) \\
	= & C \left( \frac{1}{\theta} - \frac{q}{\tau} \right) \tau^{1-\frac{1}{\theta}} g^{\frac{1}{\theta}}(a).
	\end{aligned}
	\end{equation}
	
	(i). $\tau> \theta (q+1)$. From \eqref{v''(0)} and \eqref{limv''}, we have
	$v''(0) = \lim_{r\to 0} v''(r) =0$, which shows that $v\in C^2[0, R)$.
	
	(ii). $\tau = \theta (q+1)$. From \eqref{v''(0)}, we have
	$v''(0) = C \left( \frac{g(a)}{\tau} \right)^{\frac{1}{\theta}}$.
	From \eqref{limv''}, we have
	$
	\lim\limits_{r\to 0}v''(r) = C \left( \frac{\tau}{\theta} - q \right) \tau^{-\frac{1}{\theta}} g^{\frac{1}{\theta}}(a)=C \left( \frac{g(a)}{\tau} \right)^{\frac{1}{\theta}}.
	$
	Then we get $v\in C^2[0, R)$.
	
	(iii). $\theta q < \tau < \theta (q+1)$. For $\delta \in (1, \frac{\theta}{\theta(q+1)-\tau})$, we have $-1<\frac{\delta}{\theta}[\tau - \theta(q+1)]<0$. Then the integral $\int_0^R r^{\frac{\delta}{\theta}[\tau - \theta(q+1)]}dr$ converges. By \eqref{limv''}, $\lim\limits_{r \to 0} \frac{(v''(r))^\delta}{r^{\frac{\delta}{\theta}[\tau - \theta (q+1)]}} $ is equal to a positive constant. Then the integral $\int_0^R (v''(r))^\delta dr$ converges, which shows that $v \in C^2(0, R)\cap W^{2, \delta}(0, R)$.
	
	Now the proof of this lemma is completed.
	\end{proof}
\begin{remark}	
Note that a local solution of \eqref{1.1} exists for any initial value $a\in \mathbb{R}$. If $a$ is positive, the solution $v$ is positive for all $r\in [0, R)$.
\end{remark}

The following lemma is a further property of the local solution of \eqref{1.1} when $\tau = \theta (q+\frac{1}{p-1})$ for some $p>1$, and will be used in Sections \ref{Section 4} and \ref{Section 5}.
\begin{Lemma}\label{Lemma 2.2}
Assume that $v$ is the local solution of the Cauchy problem \eqref{1.1} under the assumptions of Theorem \ref{Th 1.1}. If $\tau = \theta (q+\frac{1}{p-1})$ for some $p>1$, then
\begin{equation}\label{2.24}
\lim_{r\to 0} \left[ (v'(r))^{p-1} \right] ' = \lim_{r\to 0} \frac{(v'(r))^{p-1}}{r} = C^{l-1} \left (\frac{g(a)}{\tau}\right)^{\frac{p-1}{\theta}}.
\end{equation}
\end{Lemma}
\begin{proof}
Using \eqref{1.1} and l'H\^opital's rule, we have
\begin{equation}\label{2.25}
\lim_{r\to 0} \frac{v'(r)}{r^{\frac{1}{\theta}(\tau - \theta q)}} = C \lim_{r\to 0} \left ( \frac{\int_0^r s^{\tau-1}g(v(s))ds}{r^\tau} \right )^{\frac{1}{\theta}} = C \left( \frac{g(a)}{\tau}\right )^{\frac{1}{\theta}}.
\end{equation}
By \eqref{limv''} and \eqref{2.25}, we have
\begin{equation}\label{2.26}
\begin{aligned}
    & \lim_{r\to 0} \left[ (v'(r))^{p-1} \right] '
=   \lim_{r\to 0} \left[(p-1)(v'(r))^{p-2}v''(r) \right] \\
= &  (p-1) \left[ C\left(\frac{g(a)}{\tau}\right)^{\frac{1}{\theta}}\right]^{p-2}C\left(\frac{1}{\theta} - \frac{q}{\tau}\right)\tau^{1-\frac{1}{\theta}}g^{\frac{1}{\theta}}(a)
\lim_{r\to 0} r^{\frac{p-1}{\theta}(\tau-\theta q)-1} \\
= & C^{p-1} \left( \frac{g(a)}{\tau} \right)^{\frac{p-1}{\theta}} \frac{p-1}{\theta} (\tau-\theta q) \lim_{r\to 0}r^{\frac{p-1}{\theta}(\tau-\theta q)-1}\\
= & C^{p-1} \left (\frac{g(a)}{\tau}\right)^{\frac{p-1}{\theta}},
\end{aligned}
\end{equation}
where $\tau = \theta(q+\frac{1}{p-1})$ is used to get the last equality. By using l'H\^opital's rule and \eqref{2.26}, we have
\begin{equation}\label{2.27}
\lim_{r\to 0} \frac{(v'(r))^{p-1}}{r} = \lim_{r\to 0} \left[ (v'(r))^{p-1} \right] ' = C^{p-1} \left (\frac{g(a)}{\tau}\right)^{\frac{p-1}{\theta}},
\end{equation}
which completes the proof of \eqref{2.24}.
\end{proof}

\begin{remark}
In order to derive $\lim_{r\to 0} \frac{(v'(r))^{p-1}}{r}  = C^{p-1} \left (\frac{g(a)}{\tau}\right)^{\frac{p-1}{\theta}}$ in \eqref{2.27}, if we do not use l'H\^opital's rule directly, alternatively, we can also calculate it from \eqref{1.1} under the condition $\tau = \theta(q+\frac{1}{p-1})$.
\end{remark}


\vspace{3mm}

\section{Entire existence and nonexistence of the Cauchy problem}\label{Section 3}

\sloppy{}

In this section, we show that the Keller-Osserman condition \eqref{KO} is a necessary and sufficient condition for the existence of entire solutions of Cauchy problem \eqref{1.1} in the case when $\tau \ge \theta q+1$. In the case when $\theta q < \tau < \theta q +1$, we prove that conditions \eqref{KO} and \eqref{1.3} are necessary and sufficient condition respectively, for the entire existence of problem \eqref{1.1}.

For $v(r)$ satisfying \eqref{1.1}, a direct calculation leads to
\begin{equation}\label{3.1}
v^{\prime\prime}(r)\left( v^{\prime}(r)\right)^{\theta -1}+\frac{q}{r}\left( v^{\prime}(r)\right)^{\theta}
=\frac{C^\theta}{\theta} r^{\tau - q \theta -1} g(v(r)).
\end{equation}
Note that the Keller-Osserman condition of \eqref{1.1} is hidden in the relationship between the first  and last terms in \eqref{3.1}, which will be revealed in the proof.

Before proving Theorems \ref{Th 1.2} and \ref{Th 1.3}, we prove the following lemma, which will be useful in the proof of the necessity of Keller-Osserman condition \eqref{KO} in Theorems \ref{Th 1.2} and \ref{Th 1.3}. In this lemma, $\tau$ is only assumed to satisfy $\tau > \max \{\theta (q-1), 0\}$. 
\begin{Lemma}\label{Lem 3.1}
Assume that $g: \mathbb{R} \to (0,+\infty)$ is a non-decreasing continuous function, the constants $q\ge 0$, $\theta> 0$, $C>0$ and $\tau > \max \{\theta (q-1), 0\}$, and $v$ is a solution of the Cauchy problem \eqref{1.1} with $v'(0)=0$ for all $r\ge 0$ and for some initial value $a$.
If condition \eqref{KO} does not hold, then there exists a large positive constant $R_0$ such that
\begin{equation}\label{claim}
v^{\prime\prime}(r)\left( v^{\prime}(r)\right)^{\theta -1} > \frac{C^\theta}{2\theta} r^{\tau - q \theta -1} g(v(r)),
\end{equation}
for $r>R_0$.
\end{Lemma}
\begin{proof}[Proof of Lemma \ref{Lem 3.1}]
Letting
$\tilde g(t)=\left( \int_{0}^{t}g (s)ds\right)^{-\frac{1}{\theta+1}}$, since condition \eqref{KO} deos not hold, we have  $\int^{\infty}\tilde g(t) dt<\infty$. Hence, $\int_s^{\infty}\tilde g(t)dt\to 0$ as $s\to \infty$. In addition, by observing that $\tilde g(t)$ is decreasing in $(0, \infty)$, we have
\begin{equation*}
	0<t \tilde g(t)\leq 2\int_{\frac{t}{2}}^{t}\tilde g(s)ds
	<2\int_{\frac{t}{2}}^{\infty}\tilde g(s)ds \to 0, \quad {\rm as} \quad t\to \infty,
\end{equation*}
which shows that
\begin{equation*}
	t \tilde g(t)=t \left( \int_{0}^{t}g (s)ds\right)^{-\frac{1}{\theta+1}}\to 0,\quad {\rm as} \quad t\to \infty.
\end{equation*}
Combining this with the fact that $g$ is non-decreasing, we have
\begin{equation*}
	\left(\frac{ g^{\frac{1}{\theta}}(t)}{t} \right)^\theta =\frac{t g(t)}{t^{\theta+1}}
	\geq\frac{\int_{0}^{t}g(s)ds}{t^{\theta+1}}=(t\tilde g(t))^{-(\theta+1)}\to \infty, \quad {\rm as} \quad t\to \infty.
\end{equation*}
Therefore there exists a constant $t_1 > a$ such that 	
\begin{equation}\label{3.6'}
	g^{\frac{1}{\theta}}(t)>t-a,\quad \forall \,t>t_1.
\end{equation}
Since $v^{\prime}(r)>0$ for $r>0$, we have $v(r)>v(0)=a$ for $r>0$.
Since $g$ is non-decreasing and $v(0)=a$, we have from \eqref{1.1} that 
$$v^\prime(r) \ge Cr^{-q}g^{\frac{1}{\theta}}(a)\left (\int_0^r s^{\tau-1} ds \right)^{\frac{1}{\theta}} = C\left (\frac{1}{\tau}\right)^{\frac{1}{\theta}}r^{\frac{\tau}{\theta}-q} g^{\frac{1}{\theta}}(a).$$
Integrating from $0$ to $r$ and using $v(0)=a$ again, we have
$$v(r)\geq C\left (\frac{1}{\tau}\right)^{\frac{1}{\theta}} \frac{\theta}{\tau - q \theta + \theta}g^{\frac{1}{\theta}}(a) r^{\frac{\tau}{\theta}-q+1}  +a, \quad   r>0.$$
Since $\tau > \max \{\theta (q-1), 0\}$, from the above inequality, there exists $\tilde r>0$ such that
$$v(r)>t_1, \quad {\forall} r>\tilde r.$$
Therefore,  from \eqref{3.6'}, we have
\begin{equation}\label{3.10}
	g^{\frac{1}{\theta}}(v(r))>v(r)-a, \quad   r>\tilde r.
\end{equation}
Since $v'(r)>0$ for $r>0$, from \eqref{3.1}, we have
\begin{equation*}
v^{\prime\prime}(r)\left( v^{\prime}(r)\right)^{\theta - 1}
\le \frac{1}{\theta} C^\theta r^{\tau - q\theta -1} g(v(r)), \quad r>0.
\end{equation*}
Multiplying this by $(\theta+1)v'(r)$, we have
\begin{equation*}
(\theta+1)v^{\prime\prime}(r)\left( v^{\prime}(r)\right)^{\theta}
\le \frac{\theta + 1}{\theta} C^\theta r^{\tau - q\theta -1} g(v(r)) v^{\prime}(r), \quad r>0.
\end{equation*}
Integrating this from $0$ to $r(>\tilde r)$, we get
\begin{equation}\label{3.11}
\begin{aligned}
		\left(v^{\prime}(r)\right)^{\theta + 1} & \le \frac{\theta + 1}{\theta} C^\theta r^{\tau - q\theta -1} g(v(r))(v(r)-a) \\
		& < \frac{\theta + 1}{\theta} C^\theta r^{\tau - q\theta -1} g^{1+\frac{1}{\theta}}(v(r)), \quad r> \tilde r,
\end{aligned}
\end{equation}
where $v^\prime(0)=0$ and the monotonicity of $g$ are used in the first inequality, and \eqref{3.10} is used in the second inequality.
By direct calculation, we obtain from \eqref{3.11} that
\begin{equation}\label{3.12}
\frac{q}{r}\left( v^{\prime}(r)\right)^{\theta}< \left[ q \theta \left (\frac{\theta+1}{C\theta} \right )^{\frac{\theta}{\theta+1}}r^{-\frac{\tau - q\theta + \theta}{\theta+1}}\right ]\frac{C^\theta}{\theta} r^{\tau - q \theta -1} g(v(r)).
\end{equation}
Since $\tau > \max \{\theta (q-1), 0\}$, we can fix a large positive constant $R_0 > \tilde r$ such that
\begin{equation}\label{3.13}
q \theta \left (\frac{\theta+1}{C\theta} \right )^{\frac{\theta}{\theta+1}}r^{-\frac{\tau - q\theta + \theta}{\theta+1}} < \frac{1}{2}, \quad r>R_0.
\end{equation}
Then we have from \eqref{3.12} and \eqref{3.13} that 
\begin{equation}\label{3.14}
\frac{q}{r}\left( v^{\prime}(r)\right)^{\theta}< \frac{C^\theta}{2\theta} r^{\tau - q \theta -1} g(v(r)),\quad r>R_0.
\end{equation}
Inserting \eqref{3.14} into \eqref{3.1}, we then get \eqref{claim}.
\end{proof}

\begin{remark}
Note that in Lemma \ref{Lem 3.1}, $\frac{1}{2}$ on the right hand side of \eqref{claim} can be replaced by any constant $\xi \in (0,1)$. Precisely, under the assumptions of Lemma \ref{Lem 3.1}, if condition \eqref{KO} does not hold, for any given constant $\xi \in (0,1)$, there exists a large positive constant $\tilde R_0$ such that
\begin{equation}\label{claim'}
v^{\prime\prime}(r)\left( v^{\prime}(r)\right)^{\theta -1} > \xi \frac{C^\theta}{\theta} r^{\tau - q \theta -1} g(v(r)),
\end{equation}
for $r>\tilde R_0$.
\end{remark}

We are ready to present the proof of Theorem \ref{Th 1.2}. 
\begin{proof}[Proof of Theorem \ref{Th 1.2}]
We first prove the sufficiency of condition \eqref{KO}. Suppose that there does not exist an entire solution of the Cauchy problem \eqref{1.1} for all $r\in [0, +\infty)$. By Theorem \ref{Th 1.1}, the solution $v$ of \eqref{1.1} with $v(0)=a\ge 0$ and $v'(0)=0$ has a maximal existence interval $[0, R)$ for some $R>0$. Namely, $v(r)\rightarrow +\infty$ as $r\rightarrow R$. Since $v^{\prime}(r)>0$ for $r>0$, from \eqref{3.1}, we obtain that
\begin{equation}\label{3.2}
v^{\prime\prime}(r)\left( v^{\prime}(r)\right)^{\theta-1}
\le \frac{C^\theta}{\theta}r^{\tau-q\theta -1} g(v(r))
< \frac{C^\theta}{\theta}R^{\tau-q\theta -1} g(v(r)), \quad 0<r<R,
\end{equation}
where $\tau \ge \theta q+1$ is used in the last inequality.
Multiplying both sides of \eqref{3.2} by $(\theta+1)v^{\prime}(r)$ and integrating it from $0$ to $r$, we get
\begin{equation*}
	\left( 	v^{\prime}(r)\right) ^{\theta +1}<\frac{\theta +1}{\theta} C^\theta R^{\tau-q\theta -1} \int_{a}^{v(r)}g (s)ds,
\end{equation*}
where $v(0)=a$ and $v^{\prime}(0)=0$ are used.
Consequently, we have
\begin{equation}\label{3.3}
\left( \int_{a}^{v(r)}g(s)ds\right) ^{-\frac{1}{\theta+1}}v^{\prime}(r)< \left [ \frac{\theta +1}{\theta} C^\theta R^{\tau-q\theta -1} \right]^{\frac{1}{\theta+1}}.
\end{equation}
Integrating \eqref{3.3}  from $0$ to $R$ and and using the fact that $v(0)=a$ and $v(R)=+\infty$, we get
\begin{equation*}
	\int_{a}^{\infty}\left( \int_{a}^{v}g(s)ds\right) ^{-\frac{1}{\theta+1}}dv<\left [ \frac{\theta +1}{\theta} C^\theta R^{\tau-q\theta -1} \right]^{\frac{1}{\theta+1}}R<\infty,
\end{equation*}
which contradicts with the Keller-Osserman condition \eqref{KO} when $a\ge 0$. Therefore, under the condition \eqref{KO}, the Cauchy problem \eqref{1.1} has an entire solution $v$ with $v'(0)=0$ for all $r\ge 0$ and for nonnegative initial value $a$.

\vspace{2mm}

Next, we prove the necessity of condition \eqref{KO}. Suppose on the contrary that
\begin{equation}\label{3.4}
	\int^{\infty}\left( \int_{0}^{t}g(s)ds\right)^{-\frac{1}{\theta+1}}dt<\infty.
\end{equation}
Since $\tau \ge \theta q+1 > \max \{\theta (q-1), 0\}$, from Lemma \ref{Lem 3.1}, the inequality \eqref{claim} holds under \eqref{3.4}. Multiplying both sides of \eqref{claim} by $v^\prime(r)$ and integrating from $R_0$ to $r (>R_0)$, since $\tau \ge \theta q+1$, we have
\begin{align*}
	(v^\prime(r) )^{\theta+1}>&\frac{\theta +1}{2\theta} C^\theta R_0^{\tau -q\theta -1} \int_{v(R_0)}^{v(r)}  g(s)ds + (v^\prime(R_0) )^{\theta+1}\\
	>&\frac{\theta +1}{2\theta} C^\theta R_0^{\tau -q\theta -1} \int_{v(R_0)}^{v(r)}  g(s)ds, \quad r>R_0,
\end{align*}
that is
\begin{equation}\label{3.8}
	\left( \int_{v(R_0)}^{v(r)}g(s)ds\right)^{-\frac{1}{\theta+1}} v'(r)>\left( \frac{\theta +1}{2\theta} C^\theta R_0^{\tau -q\theta -1} \right)^\frac{1}{\theta+1}, \quad r>R_0.
\end{equation} 
Without loss of generality, we can regard $v(R_0)$ as a positive constant. Otherwise, the proof is trivial. 
Integrating \eqref{3.8} again from $R_0$ to $r(>R_0)$, we have
\begin{align}
		& \left( \frac{\theta +1}{2\theta} C^\theta R_0^{\tau -q\theta -1} \right)^\frac{1}{\theta+1} (r-R_0) \notag\\
		< & \int_{v(R_0)}^{v(r)}\left( \int_{v(R_0)}^{t}g(s)ds\right)^{-\frac{1}{\theta+1}}dt \notag\\
		 \le & \left(\int_{v(R_0)}^{2v(R_0)} + \int_{2v(R_0)}^{\infty} \right)\left( \int_{v(R_0)}^{t}g(s)ds\right) ^{-\frac{1}{\theta+1}}dt \label{3.9}\\
		 \le & v^{\frac{\theta}{\theta+1}}(R_0) g^{-\frac{1}{\theta+1}}(v(R_0))+ \int_{2v(R_0)}^{\infty}\left( \int_{0}^{\frac{t}{2}}g(s)ds\right) ^{-\frac{1}{\theta+1}}dt \notag \\
		 = & v^{\frac{\theta}{\theta+1}}(R_0) g^{-\frac{1}{\theta+1}}(v(R_0)) + 2 \int_{v(R_0)}^{\infty}\left( \int_{0}^{t}g(s)ds\right) ^{-\frac{1}{\theta+1}}dt,\notag
\end{align}
where the non-decreasing property of $g$ is used in the last inequality.
Since $r$ in \eqref{3.9} can be arbitrarily large, \eqref{3.9} contradicts with \eqref{3.4}. Thus, the proof of necessity is complete.
\end{proof}

\begin{remark}
Theorem \ref{Th 1.2} has been applied to prove the nonexistence of global radial solution (with separable variables) of a parabolic $k$-Hessian equation
$$
-u_t [S_k(D^2u)]^\alpha =1, \quad {\rm \mathbb{R}^{n+1}_{-}},
$$
for $-\frac{1}{k}<\alpha<0$; see Theorem 1.1(2) in \cite{CJ2026}.
\end{remark}

Finally, we prove Theorem \ref{Th 1.3}.
\begin{proof}[Proof of Theorem \ref{Th 1.3}]
We first prove the sufficiency of condition \eqref{1.3}. Suppose that there does not exist an entire solution of the Cauchy problem \eqref{1.1} for all $r\in [0, +\infty)$. By Theorem \ref{Th 1.1}, the solution $v$ of \eqref{1.1} with $v(0)=a\ge 0$ and $v'(0)=0$ has a maximal existence interval $[0, R)$ for some $R>0$. Namely, $v(r)\rightarrow +\infty$ as $r\rightarrow R$. Since $v^{\prime}(r)>0$ for $r>0$, from \eqref{3.1}, we obtain that
\begin{equation}\label{3.15}
v^{\prime\prime}(r)\left( v^{\prime}(r)\right)^{\theta-1}
\le \frac{C^\theta}{\theta}r^{\tau-q\theta -1} g(v(r)), \quad 0<r<R.
\end{equation}
Since $\theta q<\tau < \theta q+1$, we have $-1<\tau - \theta q -1 <0$.
Multiplying both sides of \eqref{3.15} by $(\theta+\frac{1}{\kappa})(v^{\prime}(r))^{\theta +\frac{1}{\kappa}}$ and integrating it from $0$ to $r (<R)$, we get
\begin{equation}\label{3.16}
\begin{aligned}
	    \left( 	v^{\prime}(r)\right) ^{\theta +\frac{1}{\kappa}} 
	<   & \frac{\theta +\frac{1}{\kappa}}{\theta} C^\theta \int_{0}^{r} s^{\tau-\theta q -1} g (v(s)) (v'(s))^{\frac{1}{\kappa}}ds\\
	\le & \frac{\theta +\frac{1}{\kappa}}{\theta} C^\theta \left (\int_0^r \frac{1}{s^{1-\epsilon}} ds\right)^{\frac{\theta q +1-\tau}{1-\epsilon}}
	\left( \int_a^{v(r)} g^{\kappa} (s)ds\right)^{\frac{1}{\kappa}} \\
	\le & \frac{\theta +\frac{1}{\kappa}}{\theta} C^\theta \left (\frac{R^\epsilon}{\epsilon} \right)^{\frac{\theta q +1-\tau}{1-\epsilon}}
	\left( \int_a^{v(r)} g^{\kappa} (s)ds\right)^{\frac{1}{\kappa}},
\end{aligned}
\end{equation}
where $v^{\prime}(0)=0$ is used in the first inequality, H\"older's inequality is used in the second inequality, $\kappa$ and $\epsilon$ are the same constants as those in condition \eqref{1.3}. From \eqref{3.16}, we get
\begin{equation}\label{3.17}
\left( \int_a^{v(r)} g^{\kappa} (s)ds\right)^{-\frac{1}{1+\kappa \theta}} v'(r) \le \bar C,
\end{equation}
where 
$$
\bar C = \left [ \frac{\theta +\frac{1}{\kappa}}{\theta} C^\theta \left (\frac{R^\epsilon}{\epsilon} \right)^{\frac{\theta q +1-\tau}{1-\epsilon}}  \right ]^{\frac{1}{\theta + \frac{1}{\kappa}}}.
$$
Integrating \eqref{3.17}  from $0$ to $R$ and and using the fact that $v(0)=a$ and $v(R)=+\infty$, we get
\begin{equation*}
	\int_{a}^{\infty}\left( \int_{a}^{v}g^{\kappa}(s)ds\right) ^{-\frac{1}{\kappa \theta+1}}dv<\bar C R<+\infty,
\end{equation*}
which contradicts with the Keller-Osserman condition \eqref{1.3} when $a\ge 0$. Therefore, under the condition \eqref{1.3}, the Cauchy problem \eqref{1.1} has an entire solution $v$ with $v'(0)=0$ for all $r\ge 0$ and for nonnegative initial value $a$.

Next, we prove the necessity of condition \eqref{KO}. Suppose on the contrary that \eqref{3.4} holds.
Since $\tau > \theta q > \max \{\theta (q-1), 0\}$, from Lemma \ref{Lem 3.1}, the inequality \eqref{claim} holds under \eqref{3.4}. Multiplying both sides of \eqref{claim} by $v^\prime(r)$ and integrating from $R_0$ to $r (>R_0)$, since $-1<\tau - \theta q -1 <0$, we have
\begin{equation}\label{3.18}
\begin{aligned}
	(v^\prime(r) )^{\theta+1}>&\frac{\theta +1}{2\theta} C^\theta \int_{R_0}^{r} s^{\tau -q\theta -1} g(v(s))v'(s)ds + (v^\prime(R_0) )^{\theta+1}\\
	>&\frac{\theta +1}{2\theta} C^\theta r^{\tau -q\theta -1} \int_{v(R_0)}^{v(r)}  g(s)ds, \quad r>R_0,
\end{aligned}
\end{equation}
where $v'(R_0)>0$ is used in the second inequality.
From \eqref{3.18}, we have
\begin{equation}\label{3.19}
	\left( \int_{v(R_0)}^{v(r)}g(s)ds\right)^{-\frac{1}{\theta+1}} v'(r)>
	\left (\frac{\theta +1}{2\theta} C^\theta\right)^{\frac{1}{\theta +1}} r^{\frac{\tau -q\theta -1}{\theta +1}}, 
	\quad r>R_0.
\end{equation} 
Without loss of generality, we can regard $v(R_0)$ as a positive constant. Otherwise, the proof is trivial. 
Integrating \eqref{3.19} again from $R_0$ to $r(>R_0)$, \eqref{3.9} holds with the first line replaced by
$$
\left ( \frac{\theta +1}{2\theta} C^\theta \right )^{\frac{1}{\theta}}\frac{\theta +1}{\tau - \theta q + \theta}
\left ( r^{\frac{\tau - \theta q +\theta}{\theta +1}} - R_0^{\frac{\tau - \theta q +\theta}{\theta +1}}  \right ).
$$
Since $r$ in \eqref{3.9} can be arbitrarily large and $\tau-\theta q + \theta >0$, this contradicts with \eqref{3.4}. Thus, the proof of necessity is complete.
\end{proof}

\vspace{3mm}



\section{Applications to $k$-Hessian type equations}\label{Section 4}

\sloppy{}

In this section, we first recall the $k$-Hessian operators and their radial symmetric forms. Then we apply Theorem \ref{Th 1.2} to the standard $k$-Hessian equations, degenerate or singular $k$-Hessian equations, and other extended forms of $k$-Hessian equations, respectively. We also apply Theorem \ref{Th 1.3} to these equations at the end.

The normalized $k$-Hessian operator is defined by
\begin{equation}\label{4.1}
\tilde S_k [u] = S_k^{\frac{1}{k}} [D^2u],
\end{equation}
for $k=1, \cdots, n$, and the function $S_k$ is the sum of $k$-products of eigenvalues, namely
	\begin{equation}\label{4.2}
	S_k[D^2u]:= S_k(\lambda (D^2u)) = \sum_{i_1<\cdots<i_k} \prod_{s=1}^k \lambda_{i_s}(D^2u), \quad k=1, \cdots, n.
	\end{equation}
In \eqref{4.2}, $i_1, \cdots, i_k \subset \{1, \cdots, n\}$, and $\lambda(D^2u) = (\lambda_1(D^2u), \cdots, \lambda_n(D^2u))$ denote the eigenvalues of the matrix $D^2u$. The operator $S_k$ in \eqref{4.2} is the $k$-th order elementary symmetric polynomial, which is extensively studied in fully nonlinear partial differential equations related to geometric problems; see \cite{T1995, W2009}. Note that the operator $S_k$ in \eqref{4.2} reduces to the Laplacian when $k=1$, and to the Monge-Amp\`ere operator when $k=n$. Therefore, $S_k$ ($k=1, \cdots, n$) in \eqref{4.2} are a group of operators interpolating between the Laplacian and the Monge-Amp\`ere operator.

To work in the realm of elliptic setting, we define the G\"arding's cone,
	\begin{equation}\label{4.3}
	\Gamma_k = \{ \lambda \in \mathbb{R}^n \ | \ S_i(\lambda)>0, \ i=1,\cdots, k\},
	\end{equation}
and  assume that $\lambda(D^2u)\in \Gamma_{k}$.
Here, a function $u\in C^2(\Omega)$ is called admissible in $\Omega$ if $\lambda(D^2u)\in \Gamma_{k}$ for all $x\in \Omega\subset \mathbb{R}^n$.

Assume $v\in C^2[0, \infty)$ is radially symmetric with $v'(0)=0$. Then for $r=|x|=(\sum\limits_{i=1}^n x_i^2)^{1/2}$ and  $u(x)=v(r)$, the eigenvalues of $D^2u$ are given by
	\begin{equation}\label{4.4}
	\lambda(D^2 u)(x) =
	\left\{
	\begin{aligned}
	& \left ( v''(r), \frac{1}{r}v'(r), \cdots, \frac{1}{r}v'(r) \right), & r\in (0, +\infty),\\
	& (v''(0), \cdots, v''(0)), & r=0,\ \ \ \ \  \ \ \  \
	\end{aligned}
	\right.
	\end{equation}
where $v''(0)=\lim\limits_{r\rightarrow 0} \frac{v'(r)-v'(0)}{r-0}=\lim\limits_{r\rightarrow 0} \frac{v'(r)}{r}$ is used. For convenience, we will interpret  $\frac{1}{r}v'(r)|_{r=0}=v''(0)$ in the context.
By direct calculation, we have
	\begin{equation}\label{4.5}
	S_k[D^2u]:= S_k(\lambda (D^2u)) =
	\left\{
	\begin{aligned}
	&C_{n-1}^{k-1}\left(\frac{v'(r)}{r}\right)^{k-1}\left(v''(r)+\frac{n-k}{k}\frac{v'(r)}{r}\right), & r\in (0, +\infty),\\
	& C_n^k (v''(0))^k, & r=0, \ \ \ \ \  \ \ \  \
	\end{aligned}
	\right.
	\end{equation}
and
	\begin{equation}\label{4.6}
	\tilde S_k[u] =
	\left\{
	\begin{aligned}
	&\left(C_{n-1}^{k-1}\right)^{\frac{1}{k}}\left(\frac{v'(r)}{r}\right)^{\frac{k-1}{k}}\left(v''(r)+\frac{n-k}{k}\frac{v'(r)}{r}\right)^\frac{1}{k}, & r\in (0, +\infty),\\
	& \left(C_n^k\right)^{\frac{1}{k}} v''(0), & r=0. \ \ \ \ \  \ \ \  \
	\end{aligned}
	\right.
	\end{equation}

Next, we shall consider several $k$-Hessian type equations in the following subsections. 

\subsection{Standard $k$-Hessian case}\label{Section 4.1}
We firist recall the standard $k$-Hessian equations in \cite{JB2010},
	\begin{equation}\label{4.7}
	\tilde S_k [u] = f(u), \quad {\rm in} \ \mathbb{R}^n.
	\end{equation}
In order to study the radial symmetric solutions $u(x)=v(r)$ of \eqref{4.7}, by using \eqref{4.6}, it is enough to consider
	\begin{equation}\label{4.8}
	\left(C_{n-1}^{k-1}\right)^{\frac{1}{k}}\left(\frac{v'(r)}{r}\right)^{\frac{k-1}{k}}\left(v''(r)+\frac{n-k}{k}\frac{v'(r)}{r}\right)^\frac{1}{k} = f(v(r)), \ r>0,
	\end{equation}
with $v'(0)=0$. By writing it in the divergence form and integrating both sides, it is equivalent to consider the Cauchy problem
	\begin{equation}\label{4.9}
	\left\{
	\begin{aligned}
	&v'(r)= r^{1-\frac{n}{k}}  \left(  \frac{k}{C_{n-1}^{k-1}} \int_{0}^{r} s^{n-1} f^{k} (v(s))ds \right )^{\frac{1}{k}}, \quad r\in (0, +\infty), \\
	&v(0)=a.
	\end{aligned}
	\right.
	\end{equation}
	This problem corresponds to problem \eqref{1.1} when $C=\left (k/C_{n-1}^{k-1} \right )^{\frac{1}{k}}$, $q=\frac{n}{k}-1$, $\tau =n$, $\theta =k$ and $g=f^k$.
	
	In fact, we shall also consider degenerate or singular $k$-Hessian equations with $|x|^\alpha$ and gradient terms $|Du|^\beta$ on the right hand side, namely
	\begin{equation}\label{4.10}
	\tilde S_k [u] = |x|^\alpha |Du|^\beta f(u), \quad {\rm in} \ \mathbb{R}^n,
	\end{equation}
for some constants $\alpha>-1$ and $\beta<1$. In this case, letting $u(x)=v(r)$, the corresponding ordinary differential equation is
	\begin{equation}\label{4.11}
	\left(C_{n-1}^{k-1}\right)^{\frac{1}{k}}\left(\frac{v'(r)}{r}\right)^{\frac{k-1}{k}}\left(v''(r)+\frac{n-k}{k}\frac{v'(r)}{r}\right)^\frac{1}{k} = r^\alpha |v'(r)|^\beta f(v(r)), \ r>0,
	\end{equation}
with $v'(0)=0$. Observing the divergence structure of \eqref{4.11}, we get
	\begin{equation}\label{4.12}
	v'(r)= r^{1-\frac{n}{k}}  \left( \frac{k}{C_{n-1}^{k-1}} \int_{0}^{r} s^{k\alpha +n-1} |v'(s)|^{k\beta} f^{k} (v(s))ds \right )^{\frac{1}{k}}>0, \ r>0.
	\end{equation}
Therefore, the symbol of absolute value on the right hand side of \eqref{4.11} can be removed. Then, we can derive the equivalent Cauchy problem
	\begin{equation}\label{4.13}
	\left\{
	\begin{aligned}
	&v'(r)= r^{1-\frac{n}{k}}  \left( \frac{k(1-\beta)}{C_{n-1}^{k-1}} \int_{0}^{r} s^{k(\alpha+\beta) -n\beta+n-1} f^{k} (v(s))ds \right )^{\frac{1}{k(1-\beta)}}, \quad r\in (0, +\infty), \\
	&v(0)=a.
	\end{aligned}
	\right.
	\end{equation}
This problem corresponds to problem \eqref{1.1} when $C=\left[ \frac{k(1-\beta)}{C_{n-1}^{k-1}} \right]^{\frac{1}{k(1-\beta)}} $, $q=\frac{n}{k}-1$, $\tau =k(\alpha+\beta)-n\beta+n$, $\theta = k(1-\beta)$ and $g=f^k$.

\subsection{Extensions to $p$-$k$ Hessian equations}\label{Section 4.2}
In this subsection, we consider the equations in Section \ref{Section 4.1} with the Hessian matrix $D^2u$ replaced by $D(|Du|^{p-2}Du)$, where $p> 1$. In this part, a function $u\in C^2(\Omega)$ is called admissible in $\Omega$ if $\lambda(D(|Du|^{p-2}Du))\in \Gamma_{k}$ for all $x\in \Omega\subset \mathbb{R}^n$.

Assume $v\in C^2[0, \infty)$ is radially symmetric with $v'(0)=0$. Then for $r=|x|=(\sum\limits_{i=1}^n x_i^2)^{1/2}$ and  $u(x)=v(r)$, the eigenvalues of $D(|Du|^{p-2}Du)$ are given by
	\begin{equation}\label{4.14}
	\lambda (D(|Du|^{p-2}Du)) = \left ( ((v'(r))^{p-1})', \frac{(v'(r))^{p-1}}{r}, \cdots, \frac{(v'(r))^{p-1}}{r} \right), \quad r>0.
	\end{equation}
Then by direct calculations, we have	
	\begin{equation}\label{4.15}
	S_k [D(|Du|^{p-2}Du)]=C_{n-1}^{k-1}\left(\frac{(v'(r))^{p-1}}{r}\right)^{k-1}\left(((v'(r))^{p-1})'+\frac{n-k}{k}\frac{(v'(r))^{p-1}}{r}\right), \quad r>0.
	\end{equation}
	
In place of \eqref{4.7}, we consider
	\begin{equation}\label{4.16}
	S_k^{\frac{1}{k}} [D(|Du|^{p-2}Du)] = f(u), \quad {\rm in} \ \mathbb{R}^n.
	\end{equation}
Then the corresponding ordinary differential equation is
	\begin{equation}\label{4.17}
	\left(C_{n-1}^{k-1}\right)^{\frac{1}{k}}\left(\frac{(v'(r))^{p-1}}{r}\right)^{\frac{k-1}{k}}\left(((v'(r))^{p-1})'+\frac{n-k}{k}\frac{(v'(r))^{p-1}}{r}\right)^\frac{1}{k} = f(v(r)), \ r>0,
	\end{equation}
with $v'(0)=0$. From \eqref{4.17}, we get
	\begin{equation}\label{4.18}
	\left\{
	\begin{aligned}
	&v'(r)= r^{\frac{k-n}{k(p-1)}}  \left( \frac{k}{C_{n-1}^{k-1}}  \int_{0}^{r} s^{n-1} f^{k} (v(s))ds \right )^{\frac{1}{k(p-1)}}, \quad r\in (0, +\infty), \\
	&v(0)=a.
	\end{aligned}
	\right.
	\end{equation}
	This problem corresponds to problem \eqref{1.1} when $C=\left (k/C_{n-1}^{k-1} \right )^{\frac{1}{k(p-1)}}$, $q=\frac{n-k}{k(p-1)}$, $\tau =n$, $\theta =k(p-1)$ and $g=f^k$.
	
	Moreover, we can also consider the degenerate or singular versions of equation \eqref{4.16}, namely the equation \eqref{1.5} for some constants $\alpha>-1$ and $\beta<p-1$.
By calculations on the radial symmetric solutions of \eqref{1.4}, we get the Cauchy problem
	\begin{equation}\label{4.20}
	\left\{
	\begin{aligned}
	&v'(r)= r^{\frac{k-n}{k(p-1)}}  \left( \frac{k(p-\beta-1)}{C_{n-1}^{k-1}(p-1)} \int_{0}^{r} s^{k\alpha+\frac{k-n}{p-1}\beta+n-1} f^{k} (v(s))ds \right )^{\frac{1}{k(p-\beta-1)}}, \quad r\in (0, +\infty), \\
	&v(0)=a.
	\end{aligned}
	\right.
	\end{equation}
	This problem corresponds to problem \eqref{1.1} when $C=\left[ \frac{k(p-\beta-1)}{C_{n-1}^{k-1}(p-1)} \right]^{\frac{1}{k(p-\beta-1)}}$, $q=\frac{n-k}{k(p-1)}$, $\tau =k\alpha+\frac{k-n}{p-1}\beta+n$, $\theta =k(p-\beta-1)$ and $g=f^k$.

\subsection{Keller-Osserman theorem}\label{Section 4.3}
We consider equation \eqref{1.5} and prove its Keller-Osserman theorem, which involve the equations \eqref{4.7}, \eqref{4.10} and \eqref{4.16} as special cases.

We begin with two lemmas. The first lemma is a comparison result.

\begin{Lemma}\label{Lemma 4.1}
Assume that $f: \mathbb{R} \to (0,+\infty)$ is a non-decreasing continuous function, $\alpha>-1$, $\beta<p-1$ and $p>1$.
Assume either of the following assumptions holds:
\begin{itemize}
\item [(i).]
$p\le \alpha+\beta+2$, $v(r)\in C^2[0, R)$ is a solution of \eqref{4.20} for $r\in (0, R)$ with $v'(0)=0$ and $v(r)\to +\infty$ as $r\to R$, and $u(x)\in C^2(\mathbb{R}^n)$ is an entire admissible solution of \eqref{1.5}. 
\item [(ii).]
$p> \alpha+\beta+2$, $(p-1)\alpha+\beta \ge 0$, $v(r)\in C^2(0, R)\cap W^{2,\delta}(0, R)$ for $\delta \in (1, \frac{p-\beta-1}{p-\alpha-\beta-2})$ is a solution of \eqref{4.20} for $r\in (0, R)$ with $v'(0)=0$ and $v(r)\to +\infty$ as $r\to R$, and $u(x)\in C^2(\mathbb{R}^n\backslash \{0\})\cap  W_{loc}^{2,\delta}(\mathbb{R}^n)$ for $\delta \in (1, \frac{p-\beta-1}{p-\alpha-\beta-2})$ is an entire admissible solution of \eqref{1.5}. 
\end{itemize}
Then we have
\begin{equation}\label{4.21}
u(x)\le v(|x|), \quad \forall x\in B_R.
\end{equation}
\end{Lemma}
\begin{proof} 
The assumptions $p\le \alpha+\beta+2$ and $p> \alpha+\beta+2$ and the regularity of $v$ in (i) and (ii) can be checked from $\tau \ge \theta (q+1)$ and $\theta q < \tau < \theta (q+1)$ in Theorem \ref{Th 1.1}, respectively. 
Let 
$$
L[h]:= S_k^{\frac{1}{k}} [D(|Dh|^{p-2}Dh)] = |x|^\alpha |Dh|^\beta f(h).
$$

(i).
Let $w(x)=v(r)$ for $r=|x|$, a direct calculation shows that $w(x)\in C^2(B_R)$ is an entire admissible solution of \eqref{1.5}. Namely, $L[w]=0$ in $B_R$.

If $u>w$ somewhere in $B_R$, then there is some constant $a>0$ such that $u-a$ touches $w$ from below at some interior point $x_0\in B_R$, which means $u-a-w\le 0$ in $B_R$. Then there exists $R'\in (0, R)$ such that $x_0\in B_{R'}$. Since $w(x)=v(|x|)\to +\infty$ as $x\to \partial B_R$ and $u$ is bounded in $B_R$, we can assume that $\sup\limits_{B_{R'}}(u-a-w)<0$.
By the non-decreasing property of $f$, we have
\begin{equation}\label{4.22}
L[u-a] \ge L[u] =0 = L[w], \quad {\rm in} \ B_{R'}.
\end{equation}
By the maximum principle, we have
$$
0= \sup_{B_{R'}}(u-a-v) = \sup_{\partial B_{R'}}(u-a-v) <0,
$$
which is a contradiction.

\vspace{2mm}

(ii). In this case, we only need to prove that $w(x)=v(r)$ satisfies $|Dw|^{p-2}Dw \in C^1(B_R)$.
By the definition of derivative for $|Dw|^{p-2}Dw$ at $0$, we have
\begin{equation}\label{4.23}
\begin{aligned}
   & D_i(|Dw|^{p-2}D_jw)(0) \\
= & \lim_{r\to 0} \frac{(v')^{p-1}}{r} \delta_{ij} \\
= & \lim_{r\to 0} \left[ \frac{k}{n C_{n-1}^{k-1}} f^k(a) r^{\frac{k}{p-1}[(p-1)\alpha +\beta]} \right]^{\frac{p-1}{k(p-\beta-1)}} \delta_{ij} \\
= & \left( \frac{k}{n C_{n-1}^{k-1}} f^k(a) \right)^{\frac{p-1}{k(p-\beta-1)}} \left(\lim_{r\to 0} r^{\frac{(p-1)\alpha+\beta}{p-\beta -1}} \right) \delta_{ij},
\end{aligned}
\end{equation}
for $i, j=1, 2, \cdots, n$, where $\delta_{ij}$ is the usual Kronecker delta. For $x\in B_R\backslash \{0\}$, we have
\begin{equation}\label{4.24}
D_i(|Dw|^{p-2}D_jw)(x) = \left [ (p-1)(v'(r))^{p-2}v''(r) - \frac{(v'(r))^{p-1}}{r} \right ]\frac{x_ix_j}{r^2} +  \frac{(v'(r))^{p-1}}{r}\delta_{ij}.
\end{equation}
By \eqref{lim v(r)} and \eqref{limv''}, we have
\begin{equation}\label{4.25}
v'(r) = O\left ( r^{\frac{\alpha+1}{p-\beta-1}} \right), \quad v''(r) = O \left( r^{\frac{\alpha +\beta +2 -p}{p-\beta-1}} \right), \ \ {\rm as} \ r\to 0. 
\end{equation}
Then, we have
\begin{equation}\label{4.26}
D_i(|Dw|^{p-2}D_jw)(x) = O \left( r^{\frac{(p-1)\alpha +\beta}{p-\beta-1}} \right) \left( \frac{x_ix_j}{r^2} +  \delta_{ij}\right),\ \ {\rm as} \ x\to 0. 
\end{equation}
From \eqref{4.23} and \eqref{4.26}, we have
\begin{equation}\label{4.27}
D_i(|Dw|^{p-2}D_jw)(0) = \lim_{x\to 0}D_i(|Dw|^{p-2}D_jw)(x) =0,
\end{equation}
provided $(p-1)\alpha +\beta>0$. Hence, we get $|Dw|^{p-2}Dw \in C^1(B_R)$ when $(p-1)\alpha +\beta>0$.

Next, we consider the case when $(p-1)\alpha +\beta=0$.
Note that the assumption $\tau = \theta (q+\frac{1}{p-1})$ in Lemma \ref{Lemma 2.2} with $q=\frac{n-k}{k(p-1)}$, $\tau =k\alpha+\frac{k-n}{p-1}\beta+n$ and $\theta =k(p-\beta-1)$ is just $(p-1)\alpha +\beta=0$. By applying $g=f^k$ and Lemma \ref{Lemma 2.2} in \eqref{4.24}, we then get
\begin{equation}\label{4.28}
\lim_{x\to 0}D_i(|Dw|^{p-2}D_jw)(x) = \left( \frac{k}{n C_{n-1}^{k-1}} f^k(a) \right)^{\frac{p-1}{k(p-\beta-1)}} \delta_{ij}.
\end{equation}
In the case when $(p-1)\alpha +\beta=0$, \eqref{4.23} becomes
\begin{equation}\label{4.29}
D_i(|Dw|^{p-2}D_jw)(0)=\left( \frac{k}{n C_{n-1}^{k-1}} f^k(a) \right)^{\frac{p-1}{k(p-\beta-1)}} \delta_{ij}.
\end{equation}
From \eqref{4.28} and \eqref{4.29}, we have
\begin{equation}\label{4.30}
D_i(|Dw|^{p-2}D_jw)(0) = \lim_{x\to 0}D_i(|Dw|^{p-2}D_jw)(x) =\left( \frac{k}{n C_{n-1}^{k-1}} f^k(a) \right)^{\frac{p-1}{k(p-\beta-1)}} \delta_{ij},
\end{equation}
provided $(p-1)\alpha +\beta=0$. Hence, we get $|Dw|^{p-2}Dw \in C^1(B_R)$ when $(p-1)\alpha +\beta=0$.

The remaining proof of $u(x)\le v(|x|)$ for $x\in B_R$ in (ii) is the same of that of (i). We omit its details.
\end{proof}

The next lemma shows the relationship between the entire subsolution of \eqref{1.5} and the entire solution of the Cauchy problem \eqref{4.20}.
\begin{Lemma}\label{Lemma 4.2}
Assume that $f: \mathbb{R} \to (0,+\infty)$ is a non-decreasing continuous function, $\alpha>-1$, $\beta<p-1$ and $p>1$. 
\begin{itemize}
\item[(i).] If $p\le \alpha + \beta + 2$, then \eqref{1.5} admits an entire admissible subsolution $u(x)\in C^2(\mathbb{R}^n)$ if and only if the Cauchy problem \eqref{4.20} admits a solution $v(r)\in C^2[0,\infty)$ for some initial value $a$.

\item[(ii).] If $p> \alpha + \beta + 2$ and $(p-1)\alpha+\beta \ge 0$, then \eqref{1.5} admits an entire admissible subsolution $u(x)\in C^2(\mathbb{R}^n\backslash \{0\})\cap  W_{loc}^{2,\delta}(\mathbb{R}^n)$ if and only if the Cauchy problem \eqref{4.20} admits a solution $v(r)\in C^2(0, +\infty)\cap W^{2,\delta}_{loc}(0, +\infty)$ for some initial value $a$, where $\delta \in (1, \frac{p-\beta-1}{p-\alpha-\beta-2})$.
\end{itemize}
\end{Lemma}
\begin{proof}
We only prove (i), since the proof of (ii) is the same.

(i). The sufficiency is obvious. If there exists a solution $v(r)\in C^2(0,+\infty)$ satisfying \eqref{4.20}.
Letting $u(x)=v(r)$, direct calculations show that $u(x)$ is an entire solution of \eqref{1.5}. This implies that $u(x)\in C^2(\mathbb{R}^n)$ is a required subsolution of \eqref{1.1}.

Then we prove the necessity.
Suppose there is no entire solution to the Cauchy problem \eqref{4.20}. By Theorem \ref{Th 1.1}, problem \eqref{4.20} admits a solution $v(r)$ on the maximum existence interval $[0,R)$ with $v(R)= +\infty$  and $v(0)=a$.
By Lemma \ref{Lemma 4.1}, any subsolution $u(x)$ of  \eqref{1.5} satisfies $u(x)\leq v(r)$ for $|x|\leq R$. 	
In particular, $u(0)\leq v(0)=a$.
However, since $a$ is arbitrary, we get a contradictions by choosing
$a = \frac{u(0)}{2}$ for $u(0)>0$;
$a = 2u(0)$ for $u(0)<0$;
and $a =-1$ for $u(0)=0$.
\end{proof}

With Lemma \ref{Lemma 4.2} in hand, we apply Theorem \ref{Th 1.2} to prove the Keller-Osserman theorem for the $k$-Hessian equations \eqref{1.5}.
\begin{Theorem}[Keller-Osserman theorem for $k$-Hessian type equations]\label{Th 4.1}
Assume that $f: \mathbb{R} \to (0,+\infty)$ is a non-decreasing continuous function, $\alpha\ge \frac{1}{k}-1$, $\beta<p-1$ and $p>1$. 
\begin{itemize}
\item[(i).] If $p\le \alpha+\beta +2$, then there exists an entire nontrivial admissible subsolution $u\in C^2(\mathbb{R}^n)$ of \eqref{1.5} if and only if 
\begin{equation}\label{4.31}
\int^{+\infty} \left (\int_0^t f^k(s) ds \right )^{-\frac{1}{k(p-\beta -1)+1}} dt= + \infty
\end{equation}
holds;
\item[(ii).] If $p>\alpha+\beta +2$ and $(p-1)\alpha +\beta \ge 0$, then there exists an entire nontrivial admissible subsolution $u(x)\in C^2(\mathbb{R}^n\backslash \{0\})\cap  W_{loc}^{2,\delta}(\mathbb{R}^n)$ of \eqref{1.5} for $\delta \in (1, \frac{p-\beta-1}{p-\alpha-\beta-2})$ if and only if \eqref{4.31} holds.
\end{itemize}
\end{Theorem}
\begin{proof}
We only prove (i), since the proof of (ii) is the same.

(i). We see that the corresponding Cauchy problem for \eqref{1.5} is \eqref{4.20}, which is just problem \eqref{1.1} when $C=\left[ \frac{k(p-\beta-1)}{C_{n-1}^{k-1}(p-1)} \right]^{\frac{1}{k(p-\beta-1)}}$, $q=\frac{n-k}{k(p-1)}$, $\tau =k\alpha+\frac{k-n}{p-1}\beta+n$, $\theta =k(p-\beta-1)$ and $g=f^k$.
When $p>1$, it is readily checked that $\alpha\ge \frac{1}{k}-1$ implies $\tau \ge \theta q +1$; $\beta < p-1$ implies $\theta >0$; and $p\le \alpha + \beta +2$ implies $\tau \ge \theta(q+1)$. Moreover, $C>0$ and $q\ge 0$ are automatically satisfied. Then by Theorem \ref{Th 1.2} and Lemma \ref{Lemma 4.2}, there exists an entire nontrivial subsolution $u(x)\in C^2 (\mathbb{R}^n)$ if and only if \eqref{4.31} holds.
\end{proof}

Using Lemma \ref{Lemma 4.2}, we apply Theorem \ref{Th 1.3} to prove the following theorem for the $k$-Hessian equations \eqref{1.5}.
\begin{Theorem}\label{Th 4.2}
Assume that $f: \mathbb{R} \to (0,+\infty)$ is a non-decreasing continuous function, $-1< \alpha < \frac{1}{k}-1$, $\beta<p-1$ and $p>1$. 
\begin{itemize}
\item[(a).] If $f$ satisfies 
	\begin{equation}\label{4.32}
	\int^{+\infty} \left ( \int_0^t f^{\kappa k}(s)ds \right )^{-\frac{1}{\kappa k(p-\beta -1)+1}} dt = + \infty,
	\end{equation}
where $\kappa =\frac{1-\epsilon}{k(\alpha+1) - \epsilon}$ for some constant $\epsilon\in (0, k(\alpha+1))$, then there exists an entire nontrivial admissible subsolution $u(x)\in C^1(\mathbb{R}^n)\cap C^2(\mathbb{R}^n\backslash \{0\})$ of \eqref{1.5}. Moreover, the following properties hold:
\begin{itemize}
\item[(i).] If $p\le \alpha +\beta +2$, then $u(x) \in C^2 (\mathbb{R}^n)$; 

\item[(ii).] If $p> \alpha + \beta +2$, $(p-1)\alpha + \beta \ge 0$, then $u(x)\in C^2(\mathbb{R}^n\backslash \{0\})\cap  W_{loc}^{2,\delta}(\mathbb{R}^n)$ for $\delta \in (1, \frac{p-\beta-1}{p-\alpha-\beta-2})$.
\end{itemize}

\item[(b).] If the equation \eqref{1.5} admits an entire nontrivial admissible subsolution $u$ satisfying (i) or (ii) in (a), then the condition \eqref{KO} holds.
\end{itemize}
\end{Theorem}
\begin{proof}
The proof can be done directly by combining Theorem \ref{Th 1.3} and Lemma \ref{Lemma 4.2}.
\end{proof}

\vspace{3mm}

\section{Applications to $\Pi_k$-Hessian type equations}\label{Section 5}

\sloppy{}

In this section, we introduce the $\Pi_k$-Hessian operators, which can be regarded as the counterparts of the $k$-Hessian operators in Section \ref{Section 4}. Various radial symmetric forms of  the $\Pi_k$-Hessian type equations, including the standard $k$-Hessian equations, degenerate or singular $k$-Hessian equations, and other extended forms of $k$-Hessian equations, will be studied. We show that the series of $\Pi_k$-Hessian type equations are also covered in the framework of Theorems \ref{Th 1.2} and Theorem \ref{Th 1.3}.

The normalized $\Pi_k$-Hessian operator is defined by
	\begin{equation}\label{5.1}
	\tilde \Pi_k[u] = (\Pi_k[D^2u])^{\frac{1}{C_n^k}},
	\end{equation}
where $C_n^k = \frac{n!}{k!(n-k)!}$ for $k=1, \cdots, n$, and the function $\Pi_k$ is the product of $k$-sums of eigenvalues, namely
	\begin{equation}\label{5.2}
	\Pi_k[D^2u]:= \Pi_k(\lambda (D^2u)) = \prod_{i_1<\cdots<i_k} \left ( \sum_{s=1}^k \lambda_{i_s}(D^2u) \right ), \quad k=1, \cdots, n.
	\end{equation}
In \eqref{5.2}, $i_1, \cdots, i_k \subset \{1, \cdots, n\}$, and $\lambda(D^2u) = (\lambda_1(D^2u), \cdots, \lambda_n(D^2u))$ denote the eigenvalues of the matrix $D^2u$.

In order to work in the realm of elliptic setting, we define the cone,
	\begin{equation}\label{5.3}
	\mathcal{P}_k = \{ \lambda \in \mathbb{R}^n \ | \ \sum_{s=1}^k \lambda_{i_s}>0\}
	\end{equation}
and  assume throughout this section that $\lambda(D^2u)\in \mathcal{P}_{k}$.
Here, a function $u\in C^2(\Omega)$ is called admissible in $\Omega$ if $\lambda(D^2u)\in \mathcal{P}_{k}$ for all $x\in \Omega\subset \mathbb{R}^n$.


The $\Pi_k$-Hessian operators are closely related to $k$-convex hypersurfaces in differential geometry \cite{S1986, S1987}, and are also used as main examples of the operators in the investigations of oblique boundary value problems for augmented Hessian equations in \cite{JT2017, JT2018, JT2019, T2020}.
It is clear that the operator $\Pi_k$ in \eqref{5.3} reduces to the Laplacian when $k=n$, and to the Monge-Amp\`ere operator when $k=1$. The $\Pi_k$-Hessian operators can thus be regarded as counterparts of the $k$-Hessian operators in the previous section,  which are another group of operators interpolating between the Laplacian and the Monge-Amp\`ere operator.

Assume $v\in C^2[0, \infty)$ is radially symmetric with $v'(0)=0$. Then for $r=|x|=(\sum\limits_{i=1}^n x_i^2)^{1/2}$ and  $u(x)=v(r)$, we have $u\in C^2(\mathbb{R}^n)$ and the eigenvalues of $D^2u$ are given by \eqref{4.4}.
By direct calculations, we have
	\begin{equation}\label{5.4}
	\Pi_k[D^2u]: = \Pi_k(\lambda (D^2u)) =
	\left\{
	\begin{aligned}
	&\left(k\frac{v'(r)}{r}\right)^{C_{n-1}^k}\left(v''(r)+(k-1)\frac{v'(r)}{r}\right)^{C_{n-1}^{k-1}}, & r\in (0, +\infty),\\
	& (kv''(0))^{C_n^k}, & r=0, \ \ \ \ \  \ \ \  \
	\end{aligned}
	\right.
	\end{equation}
for $k=1, \cdots, n$, where we adopt the convention that $C_{n-1}^k=0$ if $k=n$. Therefore, the normalized form
	\begin{equation}\label{5.5}
	\tilde \Pi_k[D^2u](x)=
	\left\{
	\begin{aligned}
	& \left(k\frac{v'(r)}{r}\right)^{\frac{n-k}{n}}\left(v''(r)+(k-1)\frac{v'(r)}{r}\right)^{\frac{k}{n}}, & \quad r\in (0, +\infty),\\
	& kv''(0), & \quad r=0, \ \ \ \ \  \ \ \  \
	\end{aligned}
	\right.
	\end{equation}
for $k=1, \cdots, n$.

Next, we shall consider several $\Pi_k$-Hessian type equations in the following subsections. 

\subsection{Standard $\Pi_k$-Hessian case}\label{Section 5.1}
We first consider the normalized $\Pi_k$-Hessian equations with the inhomogeneous term $f(u)$, namely
	\begin{equation}\label{5.6}
	\tilde \Pi_k [u] = f(u), \quad {\rm in} \ \mathbb{R}^n.
	\end{equation}
Using \eqref{5.5}, we can write the radial symmetric form of \eqref{5.6} as
	\begin{equation}\label{5.7}
	\left(k\frac{v'(r)}{r}\right)^{\frac{n-k}{n}}\left(v''(r)+(k-1)\frac{v'(r)}{r}\right)^{\frac{k}{n}} = f(v(r)), \quad r>0,
	\end{equation}
with $v'(0)=0$.
By observing the divergence structure in \eqref{5.7}, we get the corresponding Cauchy problem
\begin{equation}\label{5.8}
	\left\{
	\begin{aligned}
	&v'(r)=\frac{r^{1-k}}{k} \left ( \int_{0}^{r} n s^{n-1} f^{\frac{n}{k}} (v(s))ds\right)^{\frac{k}{n}}, \quad r\in (0, +\infty), \\
	&v(0)=a.
	\end{aligned}
	\right.
	\end{equation}
	This problem corresponds to problem \eqref{1.1} when $C=n^{\frac{k}{n}}/ k$, $q=k-1$, $\tau =n$, $\theta=\frac{n}{k}$ and $g=f^{\frac{n}{k}}$.

We shall also further consider the degenerate or singular $\Pi_k$-Hessian equations
	\begin{equation}\label{5.9}
	\tilde \Pi_k [u] = |x|^\alpha |Du|^\beta f(u), \quad {\rm in} \ \mathbb{R}^n,
	\end{equation}
for some constants $\alpha>-1$ and $\beta<1$. Letting $u(x)=v(r)$, we get
	\begin{equation}\label{5.10}
	\left(k\frac{v'(r)}{r}\right)^{\frac{n-k}{n}}\left(v''(r)+(k-1)\frac{v'(r)}{r}\right)^{\frac{k}{n}} = r^\alpha |v'(r)|^\beta f(v(r)), \quad r>0,
	\end{equation}
with $v'(0)=0$. Similar to \eqref{4.12}, we can also derive $v'(r)>0$ for $r>0$, so that the symbol of absolute value on the right hand side of \eqref{5.10} can be dispensed with. Therefore, a direct calculation leads to the Cauchy problem
	\begin{equation}\label{4.13}
	\left\{
	\begin{aligned}
	&v'(r)= \frac{ r^{1-k} }{k^{\frac{1}{1-\beta}}} \left(  \int_{0}^{r} n(1-\beta) s^{\frac{n}{k}(\alpha+\beta) -n\beta+n-1} f^{\frac{n}{k}} (v(s))ds \right )^{\frac{k}{n(1-\beta)}}, \quad r\in (0, +\infty), \\
	&v(0)=a.
	\end{aligned}
	\right.
	\end{equation}
This problem corresponds to problem \eqref{1.1} when $C=\left[ n(1-\beta) / k^{\frac{n}{k}}\right]^{\frac{k}{n(1-\beta)}} $, $q=k-1$, $\tau =\frac{n}{k}(\alpha+\beta)-n\beta+n$, $\theta = \frac{n(1-\beta)}{k}$ and $g=f^{\frac{n}{k}}$.

\subsection{Extensions to $p$-$\Pi_k$ Hessian equations}\label{Section 5.2}
In this subsection, we consider the equations in Section \ref{Section 5.1} with the Hessian matrix $D^2u$ replaced by $D(|Du|^{p-2}Du)$, where $p> 1$. In this part, a function $u\in C^2(\Omega)$ is called admissible in $\Omega$ if $\lambda(D(|Du|^{p-2}Du))\in {\mathcal P}_{k}$ for all $x\in \Omega\subset \mathbb{R}^n$.

Assume $v\in C^2[0, \infty)$ is radially symmetric with $v'(0)=0$. Then for $r=|x|=(\sum\limits_{i=1}^n x_i^2)^{1/2}$ and  $u(x)=v(r)$, the eigenvalues of $D(|Du|^{p-2}Du)$ are given by \eqref{4.14}. By direct calculations, we have
	\begin{equation}\label{5.12}
	\Pi_k[D(|Du|^{p-2}Du)]=\left(k\frac{(v'(r))^{p-1}}{r}\right)^{C_{n-1}^k}\left(((v'(r))^{p-1})'+(k-1)\frac{(v'(r))^{p-1}}{r}\right)^{C_{n-1}^{k-1}}, \quad r>0.
	\end{equation}

In place of \eqref{5.6}, we consider
	\begin{equation}\label{5.13}
	\left (\Pi_k[D(|Du|^{p-2}Du)] \right )^{\frac{1}{C_n^k}} = f(u), \quad {\rm in} \ \mathbb{R}^n.
	\end{equation}
Then the corresponding ordinary differential equation is
	\begin{equation}\label{5.14}
	\left(k\frac{(v'(r))^{p-1}}{r}\right)^{\frac{n-k}{n}}\left(((v'(r))^{p-1})'+(k-1)\frac{(v'(r))^{p-1}}{r}\right)^{\frac{k}{n}} = f(v(r)), \quad r>0,
	\end{equation}
with $v'(0)=0$. By observing the divergence structure in \eqref{5.7}, we get the Cauchy problem
	\begin{equation}\label{5.15}
	\left\{
	\begin{aligned}
	&v'(r)=r^{\frac{1-k}{p-1}} \left ( \int_{0}^{r} \frac{n}{k^{\frac{n}{k}}} s^{n-1} f^{\frac{n}{k}} (v(s))ds\right)^{\frac{k}{n(p-1)}}, \quad r\in (0, +\infty), \\
	&v(0)=a.
	\end{aligned}
	\right.
	\end{equation}
	This problem corresponds to problem \eqref{1.1} when $C=(n/ k^{\frac{n}{k}})^{\frac{n}{k(p-1)}}$, $q=\frac{k-1}{p-1}$, $\tau =n$, $\theta=\frac{n(p-1)}{k}$ and $g=f^{\frac{n}{k}}$.

	Moreover, we also consider the degenerate or singular version of equation \eqref{5.13}, namely the equation \eqref{1.6} for some constants $\alpha>-1$ and $\beta<p-1$.
	By calculations on the radial symmetric solutions of \eqref{1.6}, we get the Cauchy problem
	\begin{equation}\label{5.17}
	\left\{
	\begin{aligned}
	&v'(r)=r^{\frac{1-k}{p-1}} \left ( \int_{0}^{r} \frac{n(p-\beta-1)}{k^{\frac{n}{k}}(p-1)} s^{n\left[\frac{\alpha}{k}+(\frac{1}{k}-1)\frac{\beta}{p-1} +1\right]-1} f^{\frac{n}{k}} (v(s))ds\right)^{\frac{k}{n(p-\beta-1)}}, \quad r\in (0, +\infty), \\
	&v(0)=a.
	\end{aligned}
	\right.
	\end{equation}
	This problem corresponds to problem \eqref{1.1} when $C=\left [\frac{n(p-\beta-1)}{k^{\frac{n}{k}}(p-1)} \right ]^{\frac{n}{k(p-\beta-1)}}$, $q=\frac{k-1}{p-1}$, $\tau =n\left[\frac{\alpha}{k}+(\frac{1}{k}-1)\frac{\beta}{p-1} +1\right]$, $\theta=\frac{n(p-\beta-1)}{k}$ and $g=f^{\frac{n}{k}}$.

\subsection{Keller-Osserman theorem}\label{Section 5.3}

We consider equation \eqref{1.6} and prove its Keller-Osserman theorem in this subsection, which involve equations \eqref{5.6}, \eqref{5.9} and \eqref{5.13} as special cases. 

In order to consider equation \eqref{1.6}, we turn to consider the Cauchy problem \eqref{5.17}, which corresponds to \eqref{1.1} with $C=\left [\frac{n(p-\beta-1)}{k^{\frac{n}{k}}(p-1)} \right ]^{\frac{n}{k(p-\beta-1)}}$, $q=\frac{k-1}{p-1}$, $\tau =n\left[\frac{\alpha}{k}+(\frac{1}{k}-1)\frac{\beta}{p-1} +1\right]$, $\theta=\frac{n(p-\beta-1)}{k}$ and $g=f^{\frac{n}{k}}$. We observe that $q\ge 0$ is automatically satisfied when $p>1$; $C>0$ and $\theta>0$ are satisfied when $p>1$ and $\beta < p-1$; and $\tau > \theta q$ is satisfied when $\alpha >-1$. Thus, the basic assumptions for $\alpha$, $\beta$ and $p$ are 
\begin{equation}\label{5.18}
p>1, \ \ \alpha >-1,\ \ {\rm and} \ \beta < p-1.
\end{equation}
A further critical equality in Theorem \ref{Th 1.2} is $\tau=\theta (q+1)$, which is $p=\alpha + \beta +2$ for \eqref{1.6}. Then $\tau \ge \theta (q+1)$ and $\tau<\theta (q+1)$ correspond to 
\begin{equation}\label{5.19}
p \le \alpha +\beta +2, \ \ {\rm and} \ p> \alpha + \beta +2,
\end{equation}
respectively. In the case when $p> \alpha + \beta +2$, by Theorem \ref{Th 1.2}, the solution $v(r) \in C^2(0, +\infty) \cap W_{loc}^{2, \delta}(0, +\infty)$ with 
\begin{equation}\label{5.20}
\delta \in (1, \frac{p-\beta -1}{p-\alpha-\beta -2}).
\end{equation}
Also in this case, similar to Lemma \ref{Lemma 4.1} (ii), we still need to assume
\begin{equation}\label{5.21}
(p-1)\alpha + \beta \ge 0,
\end{equation}
in order to guarantee $|Dw|^{p-2}Dw$ is $C^1$ at $0$, where $w(x)=v(r)$.

Since \eqref{5.18}, \eqref{5.19}, \eqref{5.20} and \eqref{5.21} hold, Lemmas \ref{Lemma 4.1} and \ref{Lemma 4.2} still holds if \eqref{1.5} and \eqref{4.20} are replaced by \eqref{1.6} and \eqref{5.17}, respectively. Note that in the process of proving Lemma \ref{Lemma 4.1} under \eqref{1.6} and \eqref{5.17}, we still need the help of Lemma \ref{Lemma 2.2}.

Then we apply Theorem \ref{Th 1.2} to formulate the Keller-Osserman theorem for the $\Pi_k$-Hessian equation \eqref{1.6}.
\begin{Theorem}[Keller-Osserman theorem for $\Pi_k$-Hessian type equations]\label{Th 5.1}
Assume that $f: \mathbb{R} \to (0,+\infty)$ is a non-decreasing continuous function, $\alpha\ge \frac{k}{n}-1$, $\beta<p-1$ and $p>1$. 
\begin{itemize}
\item[(i).] If $p\le \alpha+\beta +2$, then there exists an entire nontrivial admissible subsolution $u\in C^2(\mathbb{R}^n)$ of \eqref{1.6} if and only if 
\begin{equation}\label{5.22}
\int^{+\infty} \left (\int_0^t f^{\frac{n}{k}}(s) ds \right )^{-\frac{1}{\frac{n(p-\beta -1)}{k}+1}} dt= + \infty
\end{equation}
holds;
\item[(ii).] If $p>\alpha+\beta +2$ and $(p-1)\alpha +\beta \ge 0$, then there exists an entire nontrivial admissible subsolution $u(x)\in C^2(\mathbb{R}^n\backslash \{0\})\cap  W_{loc}^{2,\delta}(\mathbb{R}^n)$ of \eqref{1.6} for $\delta \in (1, \frac{p-\beta-1}{p-\alpha-\beta-2})$ if and only if \eqref{5.22} holds.
\end{itemize}
\end{Theorem}
\begin{proof}
The proof follows that of Theorem \ref{Th 4.1}. Now the Keller-Osserman condition \eqref{KO} with $g=f^{\frac{n}{k}}$ and $\theta=\frac{n}{k}(p-\beta -1)$ gives the condition \eqref{5.22}.
\end{proof}

Since Lemmas \ref{Lemma 4.1} and \ref{Lemma 4.2} still holds if \eqref{1.5} and \eqref{4.20} are replaced by \eqref{1.6} and \eqref{5.17}, respectively, we also apply Theorem \ref{Th 1.3} to formulate the following theorem for the $\Pi_k$-Hessian equation \eqref{1.6}:
\begin{Theorem}\label{Th 5.2}
Assume that $f: \mathbb{R} \to (0,+\infty)$ is a non-decreasing continuous function, $-1< \alpha < \frac{k}{n}-1$, $\beta<p-1$ and $p>1$. 
\begin{itemize}
\item[(a).] If $f$ satisfies 
	\begin{equation}\label{5.23}
	\int^{+\infty} \left ( \int_0^t f^{\frac{\kappa n}{k} }(s)ds \right )^{-\frac{1}{\frac{\kappa n(p-\beta -1)}{k}+1}} dt = + \infty,
	\end{equation}
where $\kappa =\frac{1-\epsilon}{\frac{n(\alpha+1)}{k} - \epsilon}$ for some constant $\epsilon\in (0, \frac{n(\alpha+1)}{k})$, then there exists an entire nontrivial admissible subsolution $u(x)\in C^1(\mathbb{R}^n)\cap C^2(\mathbb{R}^n\backslash \{0\})$ of \eqref{1.6}. Moreover, the following properties hold:
\begin{itemize}
\item[(i).] If $p\le \alpha +\beta +2$, then $u(x) \in C^2 (\mathbb{R}^n)$; 

\item[(ii).] If $p> \alpha + \beta +2$, $(p-1)\alpha + \beta \ge 0$, then $u(x)\in C^2(\mathbb{R}^n\backslash \{0\})\cap  W_{loc}^{2,\delta}(\mathbb{R}^n)$ for $\delta \in (1, \frac{p-\beta-1}{p-\alpha-\beta-2})$.
\end{itemize}

\item[(b).] If the equation \eqref{1.6} admits an entire nontrivial admissible subsolution $u$ satisfying (i) or (ii) in (a), then the condition \eqref{KO} holds.
\end{itemize}
\end{Theorem}
\begin{proof}
The proof can be done directly by combining Theorem \ref{Th 1.3} and Lemma \ref{Lemma 4.2} (with \eqref{1.5} replaced by \eqref{1.6}).
\end{proof}

\vspace{3mm}

\section{Remarks and Examples}\label{Section 6}
In this section, we present some remarks related to the Cauchy problem \eqref{1.1} as well as the $k$-Hessian and $\Pi_k$-Hessian equations. We show some explicit examples of $f$ for \eqref{1.5} and \eqref{1.6}, and for the entire subsolutions of \eqref{1.6}.

\begin{remark}\label{Re 6.1}
When treating the partial differential equations in \cite{K1957, O1957, NU1997, JB2010, BF2022, JJL2024}, the parameter $\theta$ in the corresponding Cauchy problem \eqref{1.1} is always required to be larger than or equal to $1$. In this paper, $\theta=k$ in \eqref{4.7} and $\theta=\frac{n}{k}$ in \eqref{5.6} are both larger than or equal to $1$. Note that our $\theta$ for equations \eqref{4.10}, \eqref{4.16}, \eqref{1.5}, \eqref{5.9}, \eqref{5.13} and \eqref{1.6} may be less than $1$. For this purpose, we have already discussed the $0<\theta <1$ case separately; see Case 2 in the proof of Theorem \ref{Th 1.1}. Therefore, our result in this paper can cover all positive $\theta$.  
\end{remark}

\begin{remark}\label{Re 6.2}
Note that Theorem \ref{Th 5.1} for the $\Pi_k$-Hessian equations is new. When $k=n$, Theorem \ref{Th 5.1} reduces to the results for the classical Keller-Osserman condition of the Laplace type equation in \cite{K1957, O1957}. When $k=1$, our result agree with those in \cite{JB2010} for the Monge-Amp\`ere equations. We remark that the $k=n-1$ case also has concrete application to the $(n-1)$ Monge-Amp\`ere equation, $\det^{\frac{1}{n}}(\Delta u I - D^2u)=f(u)$, which has attracted some research interest in recent years; see \cite{JL2021, TW2016}. The Keller-Osserman condition for the $(n-1)$ Monge-Amp\`ere equation is established in \cite{JJL2024}. It is easily checked that $\det^{\frac{1}{n}}(\Delta u I - D^2u)=\tilde P_{n-1}(D^2u)$. Therefore, Theorem \ref{Th 5.1} here embraces the corresponding results in \cite{JJL2024} as a special case when $k=n-1$ and $\alpha=\beta =0$. When the integral in the Keller-Osserman condition converges, the boundary blow-up problem for the  $(n-1)$ Monge-Amp\`ere equation in a bounded domain of $\mathbb{R}^n$ is investigated in \cite{JDJ2025}. 
\end{remark}

\begin{remark}\label{Re 6.3}
The operator $S_k [D(|Du|^{p-2}Du)]$ was first introduced by the second author and X.-J. Wang in \cite{TW1999}.It reduces to the $p$-Laplacian operator when $k=1$, and to the $k$-Hessian operator when $p=2$. Therefore, the main results, Theorems \ref{Th 4.1} and \ref{Th 5.1} have similarities to those for the $p$-Laplacian equations.
For example, when we apply Theorem \ref{Th 4.1} to equation \eqref{4.16}, the case of entire subsolutions in $C^2(\mathbb{R}^n)$ is $1<p\le 2$, and the case of entire subsolutions in $C^2(\mathbb{R}^n\backslash \{0\})\cap  W_{loc}^{2,\delta}(\mathbb{R}^n)$ is $p>2$. For the particular equation \eqref{4.16}, $p>2$ case was studied in \cite{BF2022}, and the standard $k$-Hessian case when $p=2$ was considered in \cite{JB2010}. Up to our knowledge, the result here in the $1<p<2$ case for equation \eqref{4.16} is new.
\end{remark}

\begin{remark}\label{Re 6.4}
This paper applies the Cauchy problem \eqref{1.1} to two kinds of nonlinear partial differential equations, the $k$-Hessian type equations and the $\Pi_k$-Hessian type equations. Comparing the forms of $p$, $q$, $\tau$ and $\theta$ in \eqref{4.20} and \eqref{5.17}, we see the ``duality'' between these two kinds of equations. Namely, by replacing $k$ in $p$, $q$, $\tau$ and $\theta$ in \eqref{4.20} by $\frac{n}{k}$, we get the corresponding $p$, $q$, $\tau$ and $\theta$ in \eqref{5.17}; while replacing $k$ in $p$, $q$, $\tau$ and $\theta$ in \eqref{5.17} by $\frac{n}{k}$, we get the corresponding $p$, $q$, $\tau$ and $\theta$ in \eqref{4.20}. Therefore, the Keller-Osserman conditions \eqref{4.31} and \eqref{5.22} for $k$-Hessian type equations \eqref{1.5} and $\Pi_k$-Hessian type equations \eqref{1.6} respectively also have such duality. Due to this kind of duality, it is interesting to note that the $m$-Hessian type equation $S_m^{\frac{1}{m}} [D(|Du|^{p-2}Du)] = |x|^\alpha |Du|^\beta f(u)$ and the $\Pi_k$-Hessian type equation \eqref{1.6} have exactly the same Keller-Osserman condition when $km=n$. In the $km=n$ case, the more general equation,
\begin{equation}
A S_m^{\frac{1}{m}} [D(|Du|^{p-2}Du)] + B \Pi_k^{\frac{1}{C_n^k}} [D(|Du|^{p-2}Du)] = |x|^\alpha |Du|^\beta f(u)
\end{equation}
with $A(C_{n-1}^{m-1})^{\frac{1}{m}}+ B k^{1-\frac{1}{m}}>0$, has the Keller-Osseman condition 
\begin{equation}
\int^{+\infty} \left (\int_0^t f^{m}(s) ds \right )^{-\frac{1}{m(p-\beta -1)}+1} = + \infty,
\end{equation}
where $A$ and $B$ are constants.
\end{remark}

\begin{remark}\label{Re 6.5}
Theorems \ref{Th 4.1} and \ref{Th 5.1} can be extended to more general equations
\begin{equation}\label{6.3}
	S_k^{\frac{1}{k}} [D(|Du|^{p-2}Du)] = |x|^\alpha |Du|^\beta (x\cdot Du)^\gamma f(u), 
\end{equation}
and 
\begin{equation}\label{6.4}
	\left (\Pi_k[D(|Du|^{p-2}Du)] \right )^{\frac{1}{C_n^k}} = |x|^\alpha  |Du|^\beta (x\cdot Du)^\gamma f(u), 
\end{equation}
respectively. The Cauchy problems for equations \eqref{6.3} and \eqref{6.4} are \eqref{4.20} and \eqref{5.17}, with $\alpha$ and $\beta$ replaced by $\alpha+\gamma$ and $\beta+\gamma$ respectively. Thus, we can derive the extended versions of Theorems \ref{Th 4.1} and \ref{Th 5.1}.
\end{remark}

\begin{remark}\label{Re 6.6}
Note that Theorem \ref{Th 1.1} is proved when $g$ is a positive continuous function.
Alternatively, if $g$ is nonnegative and satisfies $g(0)=0$, $g(t)>0$ for $t>0$, a positive solution $v$ of \eqref{1.1} with the initial value $a\ge0$ will also be admissible. Therefore, Theorem \ref{Th 1.1} still holds in this case.
We then remark that the assumption for $g$ throughout this paper can be replaced by $g: \mathbb{R} \to [0, \infty)$ being a nonnegative and non-decreasing function satisfying $g(0)=0$, $g(t)>0$ for $t>0$. Consequently, the function $f$ in Theorems \ref{Th 4.1} and \ref{Th 5.1} also has the same alternative conditions on $g$. Moreover, in Theorems \ref{Th 4.1} and \ref{Th 5.1}, ``an entire nontrivial admissible subsolution'' should be replaced by ``a nonnegative entire nontrivial admissible subsolution'' when $a=0$, and by ``a positive entire nontrivial admissible subsolution'' when $a>0$. 
\end{remark}

Next, we gave some explicit examples of $f$ satisfying the two alternative assumptions in Remark \ref{Re 6.6}. We present them in the following two corollaries.
\begin{Corollary}\label{Co 1.1}
Let $c\geq 0$ and $f(t)=e^{c t}$, under the assumptions of Theorem \ref{Th 4.1} or Theorem \ref{Th 4.2}  (Theorem \ref{Th 5.1} or Theorem \ref{Th 5.2}), then \eqref{1.5} (or \eqref{1.6}) admits an entire nontrivial admissible subsolution  if and only if 
$$c=0.$$
\end{Corollary}

Note that Corollary \ref{Co 1.1} ($c =0$) confirms the existence of an entire subsolution $u$ of equations $S_k^{\frac{1}{k}} [D(|Du|^{p-2}Du)] = |x|^\alpha |Du|^\beta$ and $\left (\Pi_k[D(|Du|^{p-2}Du)] \right )^{\frac{1}{C_n^k}} = |x|^\alpha  |Du|^\beta$. It would be interesting to further study the Liouville type results for entire solutions to these equations in $\mathbb{R}^n$.

\begin{Corollary}\label{Co 1.2}
	Let $d\geq 0$ and
	\begin{equation*}
		f(t)=
		\left\{
		\begin{split}
			&t^d,&t>0,\\
			&0,&t\leq0,
		\end{split}
		\right.
	\end{equation*} under the assumptions of Theorem \ref{Th 4.1} or Theorem \ref{Th 4.2} (Theorem \ref{Th 5.1} or Theorem \ref{Th 5.2}).  Then \eqref{1.5} (or \eqref{1.6}) admits a nonnegative nontrivial admissible subsolution $u$ if and only if 
	$$d\leq p-\beta -1.$$
\end{Corollary}

\begin{remark}\label{Re 6.7}
Note that for equation \eqref{1.5} in Corollary \ref{Co 1.1} (equation \eqref{1.6} in Corollary \ref{Co 1.2}), the regularity of the subsolution is the same as that in Theorem \ref{Th 4.1} or Theorem \ref{Th 4.2}  (Theorem \ref{Th 5.1} or Theorem \ref{Th 5.2}).
\end{remark}

We first prove Corollary \ref{Co 1.1}, concerning exponential nonlinearities on the right hand side.
\begin{proof}[Proof of Corollary \ref{Co 1.1}]
We only present the proof for equation \eqref{1.6}. The proof for equation \eqref{1.5} is similar.

Since $c\geq0$, it is clear that $f\in C(\mathbb{R})$ is a positive and non-decreasing function. From Theorem \ref{Th 5.1}, entire subsolutions of \eqref{1.6} exist if and only if \eqref{5.22} holds, that is,
\begin{equation}\label{5.1}
	\int_a^{+\infty}\left( \int_{0}^{t}e^{\frac{nc}{k}s}ds\right)^{-\frac{1}{\frac{n(p-\beta -1)}{k}+1}}dt=+\infty,
\end{equation}where $a$ is an arbitrary constant in $\mathbb{R}$.
When $c=0$, \eqref{5.1} is obvious.
When $c>0$, we have
\begin{equation*}
	\int_a^{+\infty}\left( \int_{0}^{t}e^{\frac{nc}{k}s}ds\right) ^{-\frac{1}{\frac{n(p-\beta -1)}{k}+1}}dt=\left(\frac{k}{nc}\right)^{-\frac{1}{\frac{n(p-\beta -1)}{k}+1}}\int_a^{+\infty}(e^{\frac{nc}{k}t}-1) ^{-\frac{1}{\frac{n(p-\beta -1)}{k}+1}}dt.
\end{equation*}
Let $F(t):=(e^{\frac{nc}{k}t}-1)^{-\frac{1}{\frac{n(p-\beta -1)}{k}+1}}$ for $t\geq a$.
Since
\begin{equation*}
	\lim_{t\to +\infty}t^2F(t)=	
	\lim_{t\to +\infty}\frac{t^2}{(e^{\frac{nc}{k}t}-1)^{\frac{1}{\frac{n(p-\beta -1)}{k}+1}}}
	=\left( \lim_{t\to +\infty}\frac{t^{\frac{2n(p-\beta -1)}{k}+2}}{e^{\frac{nc}{k}t}-1}\right) ^{\frac{1}{\frac{n(p-\beta -1)}{k}+1}}=0,\quad c>0,
\end{equation*}
we obtain that the integral $\int_a^\infty F(t)dt$ converges, which shows that \eqref{5.1} fails for $c>0$.

\vspace{1mm}

In order to apply Theorem \ref{Th 5.2}, we need to check the integral in \eqref{5.23}. It is readily checked that \eqref{5.23} has the same convergence as \eqref{5.22} for $f(t)=e^{c t}$ with $c\ge 0$. Then it leads to the same conclusion as above.
\end{proof}

We now prove Corollary \ref{Co 1.2}, for power type nonlinearities on the right hand side.
\begin{proof}[Proof of Corollary \ref{Co 1.2}]
We only present the proof for equation \eqref{1.6}. The proof for equation \eqref{1.5} is the same.

It is obvious that when $d\geq0$, $f\in C[0, \infty)$ is a nonnegative and non-decreasing function satisfying $f(0)=0$, $f(t)>0$ for $t>0$.
From Theorem \ref{Th 5.1}, the entire subsolutions of \eqref{1.6} exist if and only if \eqref{5.22} holds.
We verify that $f$ satisfies the condition \eqref{5.22}, that is
\begin{equation}\label{5.2}
	\int_{a^\prime}^{+\infty}\left( \int_{0}^{t}s^{\frac{nd}{k}}ds\right) ^{-\frac{1}{\frac{n(p-\beta -1)}{k}+1}}dt=+\infty,
\end{equation}where $a^\prime$ is an arbitrary constant in $\mathbb{R^+}$.
Since $f(t)=t^d$ for $t\geq0$, we have
\begin{equation}\label{5.3}
	\int_{a^\prime}^{+\infty}\left( \int_{0}^{t}s^{\frac{nd}{k}}ds\right) ^{-\frac{1}{\frac{n(p-\beta -1)}{k}+1}}dt
	=\left(\frac{k}{nd+k}\right) ^{-\frac{1}{\frac{n(p-\beta -1)}{k}+1}}\int_{a^\prime}^{+\infty} t^{-\frac{nd+k}{n(p-\beta-1)+k}} dt.
\end{equation}
It is readily checked that the infinite integral $\int_{a^\prime}^{+\infty} t^{-\frac{nd+k}{n(p-\beta -1)+k}} dt$ in \eqref{5.3} diverges for $d\le p-\beta-1$ and converges for $d>p-\beta-1$.
Applying Theorem \ref{Th 5.1} and Remark \ref{Re 6.6}, the proof of Corollary \ref{Co 1.2} is complete for equation \eqref{1.6}.

\vspace{1mm}

In order to apply Theorem \ref{Th 5.2}, we need to check the integral in \eqref{5.23}. It is readily checked that \eqref{5.23} has the same convergence as \eqref{5.22} for $f(t)=t^d$ when $t>0$ and $f(t)=0$ when $t\le 0$. Accordingly it leads to the same conclusion as above.
\end{proof}

To end this section, we give some explicit entire subsolutions satisfying Corollaries \ref{Co 1.1} and \ref{Co 1.2}. 
These examples are for the $\Pi_k$-Hessian type equations. The examples for the $k$-Hessian equations are similar.

The first two examples of entire subsolutions are for the standard $\Pi_k$-Hessian equation \eqref{5.6}. One is for the exponential nonlinearity and the other is for the  power type nonlinearity.
\begin{Example}\label{Example 6.1}
For any $a \ge \frac{1}{k}$, $b\in \mathbb{R}^n$ and $j\in \mathbb{R}$, 
\begin{equation}\label{example1}
u(x)=\frac{1}{2}a x^2 + b\cdot x + j
\end{equation} 
is an entire admissible subsolution of \eqref{5.6} with $f(t)=e^{c t}$ with $c=0$.
\end{Example}

\begin{Example}\label{Example 6.2}
For $A\ge \frac{1}{k}$, 
\begin{equation}\label{example2}
u(x)=e^{\frac{1}{2}A|x|^2}
\end{equation}
is an entire admissible subsolution of \eqref{5.6} with $f(t)=t^d$ with $d \in [0, 1]$.
\end{Example}

The next three examples of entire subsolutions are for the $\Pi_k$-Hessian type equation \eqref{1.6}.
\begin{Example}\label{Example 6.3}
For any $a \ge \left(\frac{1}{k}\right)^{\frac{1}{3}}$ and $j\in \mathbb{R}$,
\begin{equation}\label{example3}
u(x)=\frac{1}{2}a x^2+j
\end{equation} 
is an entire admissible subsolution of \eqref{1.6} with $\alpha=2$, $\beta=-2$, $p=2$ and $f(t)=e^{c t}$ ($c=0$).
\end{Example}

\begin{Example}\label{Example 6.4}
For any $a \ge \left( (k+\frac{1}{2})C_{n-1}^{k-1}+ k C_{n-1}^k \right)^{-3}$ and $j\in \mathbb{R}$, 
\begin{equation}\label{example4}
u(x)=\frac{1}{4}a |x|^4+j
\end{equation} 
is an entire admissible subsolution of \eqref{1.6} with $\alpha=0$, $\beta=\frac{1}{6C_n^k}$, $p=\frac{3}{2}$ and $f(t)=e^{c t}$ ($c=0$).
\end{Example}

\begin{Example}\label{Example 6.5}
For any $a \ge \left( (k+\frac{9}{2})C_{n-1}^{k-1}+ k C_{n-1}^k \right)^{-\frac{1}{2}}$ and $j\in \mathbb{R}$,
\begin{equation}\label{example5}
u(x)=\frac{2}{3}a |x|^\frac{3}{2}+j
\end{equation} 
is an entire admissible subsolution of \eqref{1.6} with $\alpha=0$, $\beta=\frac{9}{C_n^k}$, $p=11$ and $f(t)=e^{c t}$ ($c=0$), where $k=1, 2 \cdots, n-1$. Note that the function $u$ in \eqref{example5} belongs to $C^2(\mathbb{R}^n\backslash \{0\})\cap W_{loc}^{2,\delta}(\mathbb{R}^n)$ for some $\delta \in (1, \frac{10C_n^k-9}{9(C_n^k-1)})$.
\end{Example}

Examples \ref{Example 6.3}, \ref{Example 6.4} and \ref{Example 6.5} satisfy $p= \alpha+\beta +2$, $p< \alpha+\beta +2$, and $p> \alpha+\beta +2$ together with $(p-1)\alpha +\beta \ge 0$, respectively. In these examples, $\alpha$ are all in the range $[\frac{k}{n}-1, +\infty)$.

Next, we show one more example with $\alpha$ in the range $(-1, \frac{k}{n}-1)$. 

\begin{Example}\label{Example 6.6}
For any $a \ge \left( n+\frac{1}{2} \right)^{-6}$ and $j\in \mathbb{R}$, 
\begin{equation}\label{example6}
u(x)=\frac{1}{4}a |x|^4+j
\end{equation} 
is an entire admissible subsolution of \eqref{1.6} with $k=n$, $\alpha=-\frac{1}{2}$, $\beta=\frac{1}{3}$, $p=\frac{3}{2}$ and $f(t)=e^{c t}$ ($c=0$).
\end{Example}



	\vspace{3mm}

    To end this section, we summarize the relationship between our examples of nonlinear partial differential equations and the Cauchy problem \eqref{1.1} in the following table \ref{Table1}.

\begin{table}[!htp]
   \centering
   {
    \linespread{1.5}  \selectfont
    
   \begin{tabular}{|m{3.8cm}|c|m{5cm}|}
        \hline
          { \ \ Nonlinear equations}&  Forms of equations&   Parameters in problem  \eqref{1.1}\\   
        \hline
         \ \ \ \ $k$-Hessian equation&   $S_k^{\frac{1}{k}} [u] = f(u)$& $C=\left (k/C_{n-1}^{k-1} \right )^{\frac{1}{k}}$, $q=\frac{n}{k}-1$, $\tau =n$, $\theta =k$, $g=f^k$\\\hline 
        
        \ \  $k$-Hessian equation with degenerate RHS&   $S_k^{\frac{1}{k}} [u] = |x|^\alpha |Du|^\beta f(u)$&  $C=\left[ \frac{k(1-\beta)}{C_{n-1}^{k-1}} \right]^{\frac{1}{k(1-\beta)}} $, $q=\frac{n}{k}-1$, $\tau =k(\alpha+\beta)-n\beta+n$, \newline $\theta = k(1-\beta)$, $g=f^k$\\\hline
        
        \ \ $p$-$k$-Hessian equation& $S_k^{\frac{1}{k}} [D(|Du|^{p-2}Du)] = f(u)$&  $C=\left (k/C_{n-1}^{k-1} \right )^{\frac{1}{k(p-1)}}$, \newline $q=\frac{n-k}{k(p-1)}$, $\tau =n$, \newline $\theta =k(p-1)$, $g=f^k$\\\hline  
        
         \ \ $p$-$k$-Hessian equation with degenerate RHS&  $S_k^{\frac{1}{k}} [D(|Du|^{p-2}Du)] = |x|^\alpha |Du|^\beta f(u)$&  $C=\left[ \frac{k(p-\beta-1)}{C_{n-1}^{k-1}(p-1)} \right]^{\frac{1}{k(p-\beta-1)}}$, \newline $q=\frac{n-k}{k(p-1)}$, $\tau =k\alpha+\frac{k-n}{p-1}\beta+n$, $\theta =k(p-\beta-1)$, $g=f^k$\\\hline
        
         \ \ $\Pi_k$-Hessian equation&  $\Pi_k^{\frac{1}{C_n^k}} [D^2u)] = f(u)$&  $C=n^{\frac{k}{n}}/ k$, $q=k-1$, $\tau =n$, $\theta=\frac{n}{k}$,  $g=f^{\frac{n}{k}}$\\\hline  
        
         \ \ $\Pi_k$-Hessian equation with degenerate RHS& $\Pi_k^{\frac{1}{C_n^k}} [D^2u)] = |x|^\alpha |Du|^\beta f(u)$&   $C=\left[ n(1-\beta) / k^{\frac{n}{k}}\right]^{\frac{k}{n(1-\beta)}}$, \newline  $q=k-1$, \newline $\tau =\frac{n}{k}(\alpha+\beta)-n\beta+n$, \newline $\theta = \frac{n(1-\beta)}{k}$, $g=f^{\frac{n}{k}}$\\\hline
         
         $p$-$\Pi_k$-Hessian equation& $\Pi_k^{\frac{1}{C_n^k}} [D(|Du|^{p-2}Du)] = f(u)$& $C=(n/ k^{\frac{n}{k}})^{\frac{n}{k(p-1)}}$, $q=\frac{k-1}{p-1}$, $\tau =n$, $\theta=\frac{n(p-1)}{k}$, $g=f^{\frac{n}{k}}$\\\hline
         
         $p$-$\Pi_k$-Hessian equation with degenerate RHS& $\Pi_k^{\frac{1}{C_n^k}} [D(|Du|^{p-2}Du)] = |x|^\alpha |Du|^\beta f(u)$& $C=\left [\frac{n(p-\beta-1)}{k^{\frac{n}{k}}(p-1)} \right ]^{\frac{n}{k(p-\beta-1)}}$, \newline $q=\frac{k-1}{p-1}$, \newline $\tau =n\left[\frac{\alpha}{k}+(\frac{1}{k}-1)\frac{\beta}{p-1} +1\right]$, \newline $\theta=\frac{n(p-\beta-1)}{k}$ , $g=f^{\frac{n}{k}}$\\\hline
   \end{tabular}
   
   }
   \vspace{2mm}
   
\caption{Nonlinear equations and the corresponding parameters in Cauchy problem \eqref{1.1}}\label{Table1}
\end{table}

\bibliographystyle{amsplain}

\end{sloppypar}
\end{document}